\documentclass[11pt]{article}
\usepackage{smile}

\usepackage{fullpage}
\usepackage{url}
\usepackage{lscape}
\usepackage{bigints}
\usepackage{framed}
\usepackage{mdframed}
\usepackage{enumerate}
\usepackage[inline]{enumitem}
\usepackage[T1]{fontenc}
\usepackage{moresize}
\usepackage{bm}
\usepackage{bbm}
\usepackage{dsfont}
\usepackage{amsmath}
\usepackage{amssymb}
\usepackage{amsthm}
\usepackage{amsfonts}
\usepackage{stmaryrd}
\usepackage{array}
\usepackage{mathrsfs}
\usepackage{mathtools} 
\usepackage{extarrows}
\usepackage{stackrel}
\usepackage{relsize,exscale}
\usepackage{scalerel}
\usepackage{colortbl}
\usepackage[nodisplayskipstretch]{setspace}
\usepackage{color}
\usepackage[usenames,dvipsnames]{xcolor}
\usepackage{cancel}
\usepackage{soul}
\usepackage{undertilde}
\usepackage{xfrac}
\usepackage{siunitx}
\usepackage{graphicx}
\usepackage{float}
\usepackage{rotating}
\usepackage{subcaption}
\usepackage{overpic}
\usepackage[all]{xy}
\usepackage{tikz}
\usetikzlibrary{arrows,matrix,positioning,calc,automata,patterns}
\usepackage{booktabs}
\usepackage{dcolumn}
\usepackage{multirow}
\usepackage{diagbox}
\usepackage{tabularx}
\usepackage{verbatim}
\usepackage{listings}
\usepackage[ruled,vlined]{algorithm2e}
\usepackage{fancyvrb}
\usepackage{hyperref}
\usepackage[round]{natbib}
\usepackage{sectsty}

\hypersetup{
    bookmarks=true,         
    unicode=false,          
    pdftoolbar=true,        
    pdfmenubar=true,        
    pdffitwindow=false,     
    pdfstartview={FitH},    
    pdftitle={My title},    
    pdfauthor={Author},     
    pdfsubject={Subject},   
    pdfcreator={Creator},   
    pdfproducer={Producer}, 
    pdfkeywords={key1, key2}, 
    pdfnewwindow=true,      
    colorlinks=true,        
    linkcolor=blue,         
    citecolor=blue,         
    filecolor=blue,         
    urlcolor=cyan           
}

\usepackage{stackengine}
\stackMath
\newcommand\tenq[2][1]{%
\def\useanchorwidth{T}%
\ifnum#1>1%
\stackunder[0pt]{\tenq[\numexpr#1-1\relax]{#2}}{\!\scriptscriptstyle\thicksim}%
\else%
\stackunder[1pt]{#2}{\!\scriptstyle\thicksim}%
\fi%
}

\makeatletter
\DeclareRobustCommand\widecheck[1]{{\mathpalette\@widecheck{#1}}}
\def\@widecheck#1#2{%
    \setbox\z@\hbox{\m@th$#1#2$}%
    \setbox\tw@\hbox{\m@th$#1%
       \widehat{%
          \vrule\@width\z@\@height\ht\z@
          \vrule\@height\z@\@width\wd\z@}$}%
    \dp\tw@-\ht\z@
    \@tempdima\ht\z@ \advance\@tempdima2\ht\tw@ \divide\@tempdima\thr@@
    \setbox\tw@\hbox{%
       \raise\@tempdima\hbox{\scalebox{1}[-1]{\lower\@tempdima\box
\tw@}}}%
    {\ooalign{\box\tw@ \cr \box\z@}}}
\makeatother

\def\tr{\mathop{\text{tr}}\kern.2ex}

\def\P{{\mathrm P}}

\def\E{{\mathrm E}}

\newcolumntype{L}[1]{>{\raggedright\let\newline\\\arraybackslash\hspace{0pt}}m{#1}}
\newcolumntype{C}[1]{>{  \centering\let\newline\\\arraybackslash\hspace{0pt}}m{#1}}
\newcolumntype{R}[1]{>{ \raggedleft\let\newline\\\arraybackslash\hspace{0pt}}m{#1}}
\newcolumntype{d}[1]{D{.}{.}{#1}}
\newcolumntype{H}{>{\setbox0=\hbox\bgroup}c<{\egroup}@{}}
\newcolumntype{Z}{>{\setbox0=\hbox\bgroup}c<{\egroup}@{\hspace*{-\tabcolsep}}}
\newcolumntype{b}{X}
\newcolumntype{s}{>{\hsize=.5\hsize}X}

\numberwithin{equation}{section}

\newtheorem{theorem}{Theorem}[section]
\newtheorem{lemma}{Lemma}[section]
\newtheorem{proposition}{Proposition}[section]
\newtheorem{assumption}{Assumption}[section]

\newtheorem{innercustomgeneric}{\customgenericname}
\providecommand{\customgenericname}{}
\newcommand{\newcustomtheorem}[2]{%
  \newenvironment{#1}[1]
  {%
   \renewcommand\customgenericname{#2}%
   \renewcommand\theinnercustomgeneric{##1}%
   \innercustomgeneric
  }
  {\endinnercustomgeneric}
}
\newcustomtheorem{customdefinition}{Definition}
\newcustomtheorem{customdefinitions}{Definitions}
\newcustomtheorem{customtheorem}{Theorem}
\newcustomtheorem{customassumption}{Assumption}
\newcustomtheorem{customlemma}{Lemma}
\newcustomtheorem{customexample}{Example}
\theoremstyle{definition}

\newtheorem{remark}{Remark}[section]

\usepackage{enumitem}
\makeatletter
\newcommand{\mylabel}[2]{#2\def\@currentlabel{#2}\label{#1}}
\makeatother

\graphicspath{{./fig3/}}

\allowdisplaybreaks

\begin{document}

\setlength{\abovedisplayskip}{5pt}
\setlength{\belowdisplayskip}{5pt}
\setlength{\abovedisplayshortskip}{5pt}
\setlength{\belowdisplayshortskip}{5pt}
\hypersetup{colorlinks,breaklinks,urlcolor=blue,linkcolor=blue}

\title{\LARGE On propensity score matching with a diverging number of matches}

\author{Yihui He\thanks{School of Mathematical Sciences, Peking University, Beijing, BJ 100080, China; e-mail: {\tt 2000010770@stu.pku.edu.cn}} ~~~and ~ Fang Han\thanks{Department of Statistics, University of Washington, Seattle, WA 98195, USA; e-mail: {\tt fanghan@uw.edu}}
}

\date{\today}

\maketitle

\vspace{-1em}

\begin{abstract}
This paper reexamines \cite{abadie2016matching}'s work on propensity score matching for average treatment effect estimation. We explore the asymptotic behavior of these estimators when the number of nearest neighbors, $M$, grows with the sample size. It is shown, hardly surprising but technically nontrivial, that the modified estimators can improve upon the original fixed-$M$ estimators in terms of efficiency. Additionally, we demonstrate the potential to attain the semiparametric efficiency lower bound when the propensity score achieves ``sufficient'' dimension reduction, 
echoing \cite{hahn1998role}'s insight about the role of dimension reduction in propensity score-based causal inference.
\end{abstract}


\section{Introduction}

Consider a quadruple $(X, W, Y(0), Y(1))$, where $W\in\{0,1\}$ denotes the treatment status, $X\in\mathbb{R}^k$ represents the pre-treatment variables, and $Y(0)$ and $Y(1)$ signify the potential outcomes \citep{neyman1923applications,rubin1974estimating} under treatment and control. This paper's primary focus is on inferring the average treatment effect (ATE), denoted as $\tau$ and defined as
\[
\tau:=\E\Big[Y(1)-Y(0)\Big],
\]
based on $N$ independent observations of $(X,W,Y(W))$. 

A significant focus of this paper centers on a nearest neighbor (NN) matching estimator \citep{abadie2006large,abadie2011bias,abadie2012martingale,abadie2016matching,lin2023estimation}; this estimator matches the subjects under study with those in the opposite treatment group who possess similar propensity scores, with these scores being estimated using the same dataset.

In detail, let $$p(x):=\P(W=1|X=x)$$ be the propensity score that was introduced in \cite{rosenbaum1983central}. We assume that $p(x)$ can be reliably quantified using a preset family of functions $\{p(x;\theta), \theta\in\Theta\subset\mathbb{R}^k\}$. The following theorem, presented informally here and to be rigorously detailed in Section \ref{sec:theory} ahead, constitutes our central result.

\begin{theorem}[Main theorem, informal]\label{thm:main-toy} Consider the estimator $\hat\tau_N(\bar\theta_N)$, which relies on propensity score-based nearest neighbor (NN) matching using estimated propensity scores $p(X_i;\bar\theta_N)$'s. Here, $\bar\theta_N$ denotes an asymptotically discrete \citep{MR1784901,van2000asymptotic} maximum likelihood estimator of $\theta$, computed from the same dataset. Then, under certain regularity conditions, the following results hold for $\hat\tau_N(\bar\theta_N)$ as the number of nearest neighbors, denoted as $M$, approaches infinity.
 \begin{enumerate}[itemsep=-.5ex,label=(\roman*)]
\item $\sqrt{N}(\hat\tau_N(\bar\theta_N)-\tau)$ is {\it approximately} asymptotically normal (Theorem \ref{thm:1});
\item the asymptotic variance of $\hat\tau_N(\bar\theta_N)$ is strictly smaller than those of a fixed $M$, and it is possible to attain the semiparametric efficiency lower bound for estimating $\tau$ \citep{hahn1998role} (an implication of Theorem \ref{thm:1});
\item there exists a consistent estimator of the asymptotic variance (Theorem \ref{thm:var}).
\end{enumerate}
\end{theorem}

\subsection{Related literature}

Our findings are anchored in the seminal contributions of Abadie and Imbens, particularly their pioneering work on propensity score nearest neighbor (NN) matching with a fixed value of $M$ as presented in \cite{abadie2016matching}, among their other influential works \citep{abadie2006large,abadie2011bias,abadie2012martingale}. What sets this paper apart from \cite{abadie2016matching} is our reevaluation of the conditions for $M$, which we force to grow to infinity as $N\to\infty$. In this context, our results also align with the recent study by \cite{lin2023estimation} and \cite{lin2022regression}, who explored the diverging-$M$ matching using the original values of $X_i$'s.

Methodologically, the estimator $\hat\tau_N(\bar\theta_N)$ is part of a broader family of propensity score-based estimators for the ATE. This family has been extensively explored in the literature, including influential works by \cite{rosenbaum1987model}, \cite{robins1994estimation}, \cite{rubin1996matching}, \cite{heckman1997matching}, \cite{hahn1998role}, \cite{scharfstein1999adjusting}, \cite{hirano2003efficient}, \cite{frolich2004finite}, \cite{frolich2005matching}, \cite{huber2013performance}, \cite{chernozhukov2018double}, \cite{su2023estimated}, among many others. Comprehensive reviews can be found in \cite{imbens2004nonparametric}, \cite{stuart2010matching}, and \cite{imbens2015matching}.

Secondly, the estimator $\hat\tau_N(\bar\theta_N)$ falls within the category of {\it substituting estimators}, where a portion of the parameters is initially estimated. A broader discussion of this class of estimators can be found in \cite{pierce1982asymptotic}, \cite{randles1982asymptotic}, \cite{pollard1989asymptotics}, \cite{newey1994large}, \cite{andrews1994empirical}, \cite{andreou2012alternative}, along with various works on causal inference \citep{robins1992estimating,henmi2004paradox,hitomi2008puzzling,lok2021estimating} and some interesting recent developments about optimal transport-based statistical inference \citep{hallin2022center,hallin2023rank}.

Thirdly, the estimator $\hat\tau_N(\bar\theta_N)$ is part of the ``graph-based statistics'' family, which aims to estimate a functional of the probability measure using random graphs constructed from an empirical realization of the underlying probability distribution. In this context, the fixed-$M$ asymptotics, as explored by Abadie and Imbens, relates interestingly to recent research on Sourav Chatterjee's rank correlation based on NN graphs with a fixed $M$ \citep{chatterjee2020new,azadkia2019simple}; in particular, they both exhibit asymptotic normality \citep{abadie2006large,lin2022limit} and bootstrap inconsistency \citep{abadie2008failure,lin2023failure}. Both of them also benefit from using a diverging $M$ for enhancing efficiency \citep{lin2023estimation,lin2021boosting}.

Theoretical underpinnings of our study differ from previous works like \cite{lin2023estimation} and \cite{lin2022regression}. Our main theorem is essentially an adaptation of the existing analysis of \cite{abadie2016matching} while addressing the limit of $M$ going to infinity throughout our proof. This adaptation is facilitated by the relative simplicity of handling estimated propensity scores, which are essentially random scalars. Like \cite{abadie2016matching}, we employ Le Cam's third lemma and Le Cam's discretization technique, as detailed in \citet[Chapter 5.7]{van2000asymptotic}, in order to avoid establishing uniform convergence results. For a justification of the use of an asymptotically discrete estimator, see \citet[Chapter 6.3]{MR1784901}.

\subsection{Notation and setup}

This paper closely follows the notation system of \cite{abadie2016matching}. Throughout, it is assumed that there exists a {\it known} form of generalized linear specification for the propensity score $p(x)$, written as 
\begin{align}\label{eq:ps-base}
p(x)=F(x'\theta^*)=:p(x;\theta^*), 
\end{align}
where $F$ is {\it known a priori} while $\theta^*$ is unknown to us. 

For any $\theta\in\Theta$, we write $\P_\theta$ to represent the joint distribution of $(X,W,Y(0),Y(1))$, with $\P(W=1|X=x)$ now taking the value $p(x;\theta):=F(x'\theta)$ while keeping the distribution of $X$ and the conditional distribution of $(Y(0),Y(1))|X,W$  unchanged. Let $\E_\theta$ and $\Var_\theta$ denote the according expectation and variance under $\P_\theta$. In particular, the data generating distribution is understood to be $\P=\P_{\theta^*}$. We similarly shorthand $\E_{\theta^*}$ and $\Var_{\theta^*}$ as ``$\E$'' and ``$\Var$''. Let 
\[
q_\theta:=\P_\theta(W=1)~~ {\rm and}~~ Y:=Y(W).
\] 

Adopting Le Cam's view on limits of experiment \citep{le1972limits}, we consider such $\theta=\theta_N$ that is allowed to change with the sample size. Accordingly, it is necessary to emphasize that, throughout this paper, we implicitly consider a triangular array setting, where sampling from a sequence of probability measures $\P_{\theta_N}$ is allowed. In this paper, the interest is in the sequence
 \begin{align*}
 \theta_N = \theta^*+h/\sqrt{N},
 \end{align*}
 with $h$ as a conformable vector of constants. Let
 \[
 Z_{N,i}=(X_{N,i},W_{N,i},Y_{N,i})
 \]
 represent the observed data, with the subscript $N$ often suppressed in the sequel. 

Following \cite{abadie2016matching}, we write 
\[
\mu(w, x):=\E[Y \mid W=w, X=x]~~ {\rm and} ~~ \sigma^{2}(w, x):=\operatorname{Var}(Y \mid W=w, X=x) 
\]
to denote the conditional mean and variance of  $Y$  given  $W=w$  and  $X=x$. Additionally, let  
\[
\bar{\mu}(w, p):=\E[Y \mid W=w, p(X)=p] ~~{\rm and}~~  \bar{\sigma}^{2}(w, p):=\operatorname{Var}(Y \mid W=w ,  p(X)=p) 
\]
represent the conditional mean and variance of  $Y$  given  $W=w$  and $p(X)=p$. Similarly, define
\[
\bar{\mu}_{\theta}(w,p):=\E_{\theta}\left[Y \mid W=w, p(X;\theta)=p\right]~~ {\rm and}~~ \bar{\sigma}^2_{\theta}(w,p):=\operatorname{Var}_{\theta}\left(Y \mid W=w, p(X;\theta)=p\right).
\]

\section{Propensity score matching}

In order to describe Abadie and Imbens's propensity score matching estimator, let's first introduce some statistics about the $M$-NN matching based on the values of a general propensity score ``estimate'' $p(X_i;\theta)$. 

Let $\mathcal{J}_{M,\theta}(i)$ represent the index set of the $M$ matches of unit $i$, measured based on the values of $p(X_i;\theta)$, among the units whose $W=1-W_{i}$. In other words, define 
\begin{align*}
&\mathcal{J}_{M,\theta}(i):=\notag\\
& \Big\{j: W_{j}=1-W_{i},\sum_{k: W_{k}=1-W_{i}} \ind\Big(\left|p(X_{i};\theta)-p(X_{k};\theta)\right| \leq\left|p(X_{i};\theta)-p(X_{j};\theta)\right|\Big) \leq M\Big\},
\end{align*}
with $\ind(\cdot)$ representing the indicator function. Furthermore, introduce 
\[
K_{M,\theta}(\cdot):\Big\{1,\cdots, N\Big\}\rightarrow \mathbb{N}~~~\text{(the set of natural numbers)}
\]
to be the number of matched times of unit $i$, i.e.,
\begin{align*}
&K_{M,\theta}(i):=\notag\\
&\sum_{1\leq j\leq N,W_{j}=1-W_{i}} \ind\left(\sum_{1\leq k\leq N,W_{k}=W_{i}}\ind\Big(\left|p(X_{k};\theta)-p(X_{j};\theta)\right|\leq\left|p(X_{i};\theta)-p(X_{j};\theta)\right|\Big)\leq M\right).
\end{align*}

The propensity score matching estimator of $\tau$, based on $p(X_i;\theta)$'s with a generic $\theta$, is then defined to be
\begin{align*}
    \widehat{\tau}_N(\theta)&=\frac{1}{N} \sum_{i=1}^{N}\left(2 W_{i}-1\right)\left(Y_{i}-\frac{1}{M} \sum_{j \in \mathcal{J}_{M,\theta}(i)} Y_{j}\right)   \notag\\
    &=\frac{1}{N} \sum_{i=1}^N (2W_i-1)\Big(1+\frac{K_{M,\theta}(i)}{M}\Big)Y_i.
\end{align*}


To complete the estimation process, one needs an estimate of $\theta^*$ to be plugged into $\hat\tau_N(\cdot)$. In light of \eqref{eq:ps-base} and following \cite{abadie2016matching}, we estimate $\theta^*$ by maximizing the log-likelihood function,
\[
L(\theta~|~Z_1,\ldots,Z_N):=\sum_{i=1}^N\Big[W_i\log F(X_i'\theta) + (1-W_i)\log\{1-F(X_i'\theta)\}\Big],
\]
yielding the maximum likelihood estimator (MLE) $\hat\theta_N$. 

The final estimator of $\tau$ is then defined to be $\hat\tau_N(\hat\theta_N)$.

\section{Theory}\label{sec:theory}

We first outline the main assumptions needed for establishing Theorem \ref{thm:main-toy}.

\begin{assumption} \label{assumption:1} It is assumed that
 \begin{enumerate}[itemsep=-.5ex,label=(\roman*)]
\item $(Y(0),Y(1))$ is independent of $W$ conditional on $X$ almost surely;
\item $p(X)$ is supported over $[\underline{p}, \bar{p}]$ with $0<\underline{p}<\bar{p}<1$,  and has a Lebesgue density that is continuous over  $[\underline{p}, \bar{p}]$. 
\end{enumerate}
\end{assumption}

\begin{assumption} \label{assumption:2} $\{Z_{i}\}_{i=1}^N$ are $N$ independent draws from $\P=\P_{\theta^*}$.
\end{assumption} 

\begin{assumption} \label{assumption:3} It is assumed that
 \begin{enumerate}[itemsep=-.5ex,label=(\roman*)]
\item $\theta^{*} \in \operatorname{int}(\Theta)$  with a compact $\Theta$, $X$ has a bounded support, and  $\E\left[X X^{\prime}\right]$  is nonsingular;
\item $F: \mathbb{R} \to(0,1)$  is continuously differentiable with a strictly positive and continuous derivative function $f$;
\item there exists a component of $X$ that is continuously distributed, has nonzero coefficient in $\theta^*$, and admits a continuous density function conditional on the rest of $X$;
\item there exist some constants $\varepsilon>0$ and $C_{\bar{\mu}}<\infty$ such that, for all  $\theta$  with the Euclidean distance $\left\|\theta-\theta^{*}\right\| \leq \varepsilon$ and $w\in\{0,1\}$, we have
 \begin{enumerate}[itemsep=-.5ex,label=(\alph*)]
\item $|\bar{\mu}_{\theta}(w,p_1)-\bar{\mu}_{\theta}(w,p_2)|\leq C_{\bar{\mu}}|p_1-p_2|$ holds for any $p_1$ and $p_2$;
\item $\bar{\sigma}^2_{\theta}(w,p)$ are equicontinuous in $p$; 
\item $\E_{\theta}\left[Y^4 \mid W=w, p(X;\theta)=p\right]$ are uniformly bounded. 
\end{enumerate}
\end{enumerate}
\end{assumption}

\begin{assumption} \label{assumption:4} For any  $\mathbb{R}^{k+2} \to \mathbb{R}$ bounded and measurable function, $r(y, w, x)$, that is continuous in  $x$, and for any sequence $ \tilde{\theta}_{N} \rightarrow \theta^{*}$, it is assumed that
\[
\E_{\tilde{\theta}_{N}}\left[r(Y, W, X) \mid W, p(X;\tilde{\theta}_N)\right] \text{ converges to } \E\left[r(Y, W, X) \mid W, p(X)\right],  \text{ almost surely}, 
\]
under $\operatorname{P}_{\tilde{\theta}_{N}}$.
\end{assumption}

Assumptions \ref{assumption:1}-\ref{assumption:4} are modified and reordered versions of Assumptions 1-5 in \cite{abadie2016matching}. In particular, Assumption \ref{assumption:1} is identical to Assumption 1 combined with Assumption 2(i) of \cite{abadie2016matching}. Assumption \ref{assumption:2} is Assumption 3 in \cite{abadie2016matching}; notably speaking, both allow for a triangular array setting intrinsically.  Assumption \ref{assumption:3} is a modified version of Assumption 4 in \cite{abadie2016matching}. Compared to \cite{abadie2016matching}, we additionally require a global Lipschitz constant in Assumption \ref{assumption:3}(iv) in order to prove Lemma \ref{lem:3} ahead. Also, we require $f$ to be continuous and $\bar{\sigma}_\theta^2$ to be equicontinuous in order to prove the uniform convergence of Riemann's Sum in Lemma \ref{lem:5} ahead. Besides, we require the uniform convergence of conditional fourth moment of $Y$ as in \cite{abadie2006large}, in order to prove the convergence of the variance estimator. Assumptions 2(ii) and (iii) in \cite{abadie2016matching} are the corollary of Assumption \ref{assumption:3}(iv). Lastly, Assumption \ref{assumption:4} is exactly Assumption 5 in \cite{abadie2016matching}.

Next, we formally introduce Le Cam's discretization trick. For any positive constant $d$, following \cite{abadie2016matching}, we transform $\hat\theta_N$ to an {\it asymptotically discrete} estimator, $\bar\theta_N$, by setting $\bar{\theta}_{N, j}=(d / \sqrt{N})\left[\sqrt{N} \widehat{\theta}_{N, j} / d\right]$, with $[.]$ as a function outputting the input's nearest integer. 

The following is the main theorem of this paper.

\begin{theorem}\label{thm:1}
Assume Assumptions \ref{assumption:1}-\ref{assumption:4} hold. Assume further that $M=O(N^v)$ for some $v<1/2$ and $M\to\infty$ as $N\to\infty$. Then, the true distribution satisfies that
\[
\lim _{d \downarrow 0} \lim _{N \rightarrow \infty} \operatorname{P}\left(\sqrt{N}\left(\sigma^{2}-c^{\prime} I_{\theta^{*}}^{-1} c\right)^{-1 / 2}\left(\widehat{\tau}_{N}\left(\bar{\theta}_{N}\right)-\tau\right) \leq z\right)=\int_{-\infty}^{z} \frac{1}{\sqrt{2 \pi}} \exp \left(-\frac{1}{2} x^{2}\right) {\sf d} x,
\]
where
\[
I_{\theta}=\E\left[\frac{f\left(X^{\prime} \theta\right)^{2}}{p(X;\theta)\left(1-p(X;\theta)\right)} X X^{\prime}\right]
\]
is the Fisher information at $\theta$, 
\[
\ c:= \E\left[\left(\frac{\operatorname{Cov}\left(X, \mu(1, X) \mid p(X)\right)}{p(X)}+\frac{\operatorname{Cov}\left(X, \mu(0, X) \mid p(X)\right)}{1-p(X)}\right) f(X'\theta^*)\right]
\]
signifies the cross term, and 
\begin{align*}
\sigma^{2} :=  \E\left[(\bar{\mu}(1, p(X))-\bar{\mu}(0, p(X))-\tau)^{2}\right]+\E\left[\frac{\bar{\sigma}^{2}(1, p(X))}{p(X)}\right]+\E\left[\frac{\bar{\sigma}^{2}(0, p(X))}{1-p(X)}\right]
\end{align*}
is the asymptotic variance of $\hat\tau_N(\theta^*)$.
\end{theorem}

It is straightforward to check that the asymptotic variance of $\hat\tau_N(\bar\theta_N)$, of the form $\sigma^2-c'I_{\theta^*}^{-1}c$, is {\it strictly smaller} than those matching estimators with a fixed $M$; cf. \citet[Proposition 1 and Theorem 1]{abadie2016matching}. In addition, the asymptotic variance is no greater than that of $\hat\tau_N(\theta^*)$, the matching estimator using the oracle $\theta^*$, a well known phenomenon. Furthermore, when 
\begin{align}\label{eq:sdr}
\bar{\mu}(w, p(x))=\mu(w, x)~~~{\rm and}~~~\bar{\sigma}^{2}(w, p(x))=\sigma^{2}(w, x),~~{\rm for}~w=0,1,
\end{align}
the asymptotic variance of $\hat\tau_N(\bar\theta_N)$ attains the following semiparametric efficiency lower bound for estimating $\tau$ \citep{hahn1998role},
\[
\sigma^{2,{\rm eff}}:= \E\left[(\mu(1, X)-\mu(0,X)-\tau)^{2}\right]+\E\left[\frac{{\sigma}^{2}(1, X)}{p(X)}\right]+\E\left[\frac{\sigma^{2}(0, X)}{1-p(X)}\right].
\]
Equation \eqref{eq:sdr}, which could be called the {\it sufficient dimension reduction condition} for propensity score matching, would be satisfied if both $\mu(w,x)$ and $\sigma^2(w,x)$ are functions of $p(x)$. Of note, under Assumption \ref{assumption:3}(ii),  Condition \eqref{eq:sdr} is equivalent to assuming that $\mu(w,x)$ and $\sigma^2(w,x)$ exhibit a single-index form as functions of $x'\theta^*$ for each $w$; cf. \cite{hahn1998role} and \citet[Chapter 3.3.2]{angrist2009mostly}.  


We next consider estimating the asymptotic variance. Inspired by \cite{abadie2016matching}, we can estimate $\sigma^2-c'I_{\theta^*}c$ using $\widehat{\sigma}^{2}-\widehat{c}'\widehat{I}_{\theta^{*}}\widehat{c}$, where
\begin{align*}
\widehat{\sigma}^{2}:= & \frac{1}{N} \sum_{i=1}^{N}\left(\left(2 W_i-1\right)\left(Y_i-\frac{1}{M} \sum_{j \in \mathcal{J}_{M}(i, \widehat{\theta}_N)} Y_j\right)-\widehat{\tau}_{N}(\widehat{\theta}_N)\right)^{2}\notag\\
&+\frac{1}{N} \sum_{i=1}^{N}\left(\left(\frac{K_{M, \widehat{\theta}_N}(i)}{M}\right)^{2}+\frac{2M-1}{M}\left(\frac{K_{M, \widehat{\theta}_N}(i)}{M}\right)\right) \widehat{\bar{\sigma}}^{2}\left(W_i, p(X_i)\right),\\
\widehat{c}:= & \frac{1}{N} \sum_{i=1}^{N}\left(\frac{\widehat{\operatorname{Cov}}\left(X_i, \mu\left(1, X_i\right) \mid p(X_i)\right)}{p(X_i;\widehat{\theta}_N)}+\frac{\widehat{\operatorname{Cov}}\left(X_i, \mu\left(0, X_i\right) \mid p(X_i)\right)}{1-p(X_i;\widehat{\theta}_N)}\right) f\left(X_i^{\prime} \widehat{\theta}_{N}\right),\\
~~{\rm and}~~\widehat{I}_{\theta^{*}}:=&\frac{1}{N} \sum_{i=1}^{N} \frac{f\left(X_i^{\prime} \widehat{\theta}_{N}\right)^{2}}{p(X_i;\widehat{\theta}_N)\left(1-p(X_i;\widehat{\theta}_N)\right)} X_i X_i^{\prime},
\end{align*}
with $\widehat{\bar{\sigma}}^{2}\left(W_i, p(X_i)\right)$ and $\widehat{\operatorname{Cov}}\left(X_i, \mu\left(w, X_i\right) \mid p(X_i)\right)$ defined in the same way as \cite{abadie2016matching}:
\begin{align*}
&\widehat{\bar{\sigma}}^{2}\left(W_i, p(X_i)\right):=\frac{1}{Q-1} \sum_{j \in \mathcal{H}_{Q}\left(i, \widehat{\theta}_N\right)}\left(Y_j-\frac{1}{Q} \sum_{k \in \mathcal{H}_{Q}\left(i, \widehat{\theta}_N\right)} Y_k\right)^{2},\\
&\widehat{\operatorname{Cov}}\left(X_i, \mu\left(W_i, X_i\right) \mid p(X_i)\right) \notag\\
&~~~~~~:=\frac{1}{L-1} \sum_{j \in \mathcal{H}_{L}\left(i, \widehat{\theta}_{N}\right)}\left(X_j-\frac{1}{L} \sum_{k \in \mathcal{H}_{L}\left(i, \widehat{\theta}_{N}\right)} X_k\right)\left(Y_j-\frac{1}{L} \sum_{k \in \mathcal{H}_{L}\left(i, \widehat{\theta}_{N}\right)} Y_k\right),\\
{\rm and}~~~&\widehat{\operatorname{Cov}}\left(X_i, \mu\left(1-W_i, X_i\right) \mid p(X_i)\right) \notag\\
&~~~~~~:=\frac{1}{L-1} \sum_{j \in \mathcal{J}_{L}\left(i, \widehat{\theta}_{N}\right)}\left(X_j-\frac{1}{L} \sum_{k \in \mathcal{J}_{L}\left(i, \widehat{\theta}_{N}\right)} X_k\right)\left(Y_j-\frac{1}{L} \sum_{k \in \mathcal{J}_{L}\left(i, \widehat{\theta}_{N}\right)} Y_k\right).
\end{align*}
Here, for a generic $\theta\in\Theta$, $\mathcal{H}_{M,\theta}(i)$ represents the set of the $M$ matches of unit $i$, based on the propensity score matching with $p(X_{i};\theta)$, among the units whose $W=W_{i}$. In other words, there is
\begin{align*}
&\mathcal{H}_{M,\theta}(i):= \notag\\
&\left\{j: W_{j}=W_{i},\sum_{k: W_{k}=W_{i}}\ind\Big(\left|p(X_{i};\theta)-p(X_{k};\theta)\right| \leq \left|p(X_{i};\theta)-p(X_{j};\theta)\right|\Big) \leq M\right\}.
\end{align*}

It is noteworthy that the variance estimator under examination is exactly the one presented in \cite{abadie2016matching} in practical applications. However, the following theorem distinguishes itself by investigating distinct asymptotic behaviors for the parameter $M$, forcing it to diverge.

\begin{theorem}\label{thm:var}
    Assume Assumptions \ref{assumption:1}-\ref{assumption:4} hold. Assume further that there exists a constant $v<1/2$ such that, as $N\to\infty$,
\[
    M\to\infty,~ M=O(N^v),~ Q\to\infty,~ Q=O(N^v),~ \text{ and } L\geq 2 \text{ is a fixed finite positive integer.} 
\]    
Then, under $\P$, $(\widehat{\sigma}^2,\widehat{c},\widehat{I}_{\theta^*})$ is a consistent estimator of $(\sigma^2,c,I_{\theta^*})$.
\end{theorem}

\begin{remark}
Theorem \ref{thm:var} differentiates from the one presented in  \cite{abadie2016matching} by incorporating asymptotically diverging values for both $M$ and $Q$. The introduction of a diverging $M$ is a necessity, as our analysis relies on $M$ approaching infinity. The introduction of a diverging $Q$, on the other hand, is convenient for managing the term $I_{6,2}$ in the proof of Theorem \ref{thm:var}. In fact, even for the study of finite-$M$ matching \citep{abadie2016matching}, one could still employ a diverging $Q$ to consistently estimate the variance.
\end{remark}

\section{Simulation}\label{sec:simulation}
This section uses simulation to complement the theory. 
We run our simulations under the following two designs, where the true ATE are both 5.

{\it Design 1.} We adopt the first design of \cite{abadie2016matching} and consider two covariates, $X_{1}$  and  $X_{2}$, that are independently sampled from the uniform distribution over $[-1 / 2,1 / 2]$. We generate potential outcomes as 
\[
Y(0)=3 X_{1}-3 X_{2}+U_{0} ~~~{\rm and}~~~ Y(1)=5+5 X_{1}+X_{2}+U_{1}, 
\]
where  $U_{0}$  and  $U_{1}$  are independent standard Gaussian random variables sampled independently of the system. Lastly,  we generate $W$ through a logistic regression model,
\[
\operatorname{P}\left(W=1 \mid X_{1}=x_{1}, X_{2}=x_{2}\right)=\frac{\exp \left(x_{1}+2 x_{2}\right)}{1+\exp \left(x_{1}+2 x_{2}\right)}.
\]
In this design, the square root of the efficiency lower bound for estimating ATE \citep{hahn1998role}, $\sigma^{\rm eff}$, is $2.473$. 

{\it Design 2:} Everything is identical to Design 1, except that now we generate potential outcomes as follows,
\[
Y(0)=2X_{1}+4X_{2}+U_{0} ~~~{\rm and}~~~  Y(1)=5-X_{1}-2X_{2}+U_{1}.
\]
In this design, the square root of the efficiency lower bound for estimating ATE \citep{hahn1998role}, $\sigma^{\rm eff}$, is $2.863$. Of note, in this setting, the asymptotic variance of $\hat\tau(\bar\theta_N)$ achieves the efficiency bound. 

We run simulations for $N=512,1024,2048,4096,8192$ and consider all $M=2^{j}, 0\leq j\leq \lfloor \frac{\log_2(N)}{2}\rfloor$, with $\lfloor \cdot\rfloor$ representing the floor function. The simulation is replicated for 2000 times, and we report the according mean absolute error (MAE), the root mean squared error (RMSE), the $95\%$ and $90\%$ coverage rates, and the normalized (scaled by $\sqrt{N}$) standard deviation (NSD) of $\hat\tau_N(\hat\theta_N)$, with $\hat\theta_N$ estimated by the standard logistic regression. For constructing the confidence interval, we use the variance estimator with $L$ and $Q$ set to be 4 and $\left[N^{1/3}\right]$, respectively.

Tables \ref{tab:1} and \ref{tab:2} display all the simulation results. It is ready to check that, in both cases, the RMSE and MAE could decrease as $M$ increases, while still maintaining a reasonable coverage rates. In addition, in Design 2, the NSD is closely matching the efficiency lower bound. All these empirical results support our theoretical findings.



The online codes to replicate all our simulation results are available on \\
\url{https://github.com/yihui-he2001/Propensity-score-matching-with-diverging-M}.


\begin{table}
\centering
\caption{Simulation results under Design 1, $\sigma^{\rm eff}=2.473$.}
\begin{tabular}{ccccccc}
  \hline
$N$ & $M$ & RMSE & MAE & $95\%$ Coverage & $90\%$ Coverage & NSD \\ 
  \hline
  512 & 1 & 0.141 & 0.112 & 0.950 & 0.901 & 3.189 \\ 
  & 2 & 0.126 & 0.100 & 0.951 & 0.902 & 2.856 \\ 
  & 4 & 0.120 & 0.095 & 0.948 & 0.899 & 2.714 \\ 
  & 8 & 0.117 & 0.092 & 0.945 & 0.895 & 2.631 \\ 
  & 16 & 0.118 & 0.093 & 0.934 & 0.887 & 2.591 \\
  \hline
  1024 & 1 & 0.100 & 0.080 & 0.952 & 0.901 & 3.212 \\ 
  & 2 & 0.089 & 0.071 & 0.952 & 0.907 & 2.851 \\ 
  & 4 & 0.084 & 0.067 & 0.955 & 0.905 & 2.689 \\ 
  & 8 & 0.082 & 0.065 & 0.953 & 0.900 & 2.601 \\ 
  & 16 & 0.081 & 0.064 & 0.947 & 0.892 & 2.561 \\ 
  & 32 & 0.085 & 0.067 & 0.938 & 0.875 & 2.555 \\
  \hline
  2048 & 1 & 0.072 & 0.058 & 0.950 & 0.896 & 3.252 \\ 
  & 2 & 0.064 & 0.052 & 0.953 & 0.898 & 2.911 \\ 
  & 4 & 0.060 & 0.048 & 0.949 & 0.906 & 2.721 \\ 
  & 8 & 0.058 & 0.046 & 0.956 & 0.905 & 2.606 \\ 
  & 16 & 0.057 & 0.045 & 0.951 & 0.900 & 2.559 \\ 
  & 32 & 0.057 & 0.046 & 0.951 & 0.890 & 2.528 \\
  \hline
  4096 & 1 & 0.051 & 0.041 & 0.950 & 0.890 & 3.239 \\ 
  & 2 & 0.046 & 0.037 & 0.946 & 0.887 & 2.953 \\ 
  & 4 & 0.043 & 0.035 & 0.946 & 0.898 & 2.773 \\ 
  & 8 & 0.042 & 0.033 & 0.946 & 0.902 & 2.665 \\ 
  & 16 & 0.041 & 0.033 & 0.950 & 0.892 & 2.618 \\ 
  & 32 & 0.041 & 0.033 & 0.941 & 0.891 & 2.599 \\ 
  & 64 & 0.042 & 0.034 & 0.933 & 0.873 & 2.576 \\
  \hline
  8192 & 1 & 0.036 & 0.029 & 0.948 & 0.896 & 3.269 \\ 
  & 2 & 0.033 & 0.026 & 0.946 & 0.889 & 2.950 \\ 
  & 4 & 0.030 & 0.024 & 0.949 & 0.890 & 2.735 \\ 
  & 8 & 0.030 & 0.024 & 0.945 & 0.890 & 2.678 \\ 
  & 16 & 0.029 & 0.023 & 0.949 & 0.892 & 2.621 \\ 
  & 32 & 0.029 & 0.023 & 0.948 & 0.891 & 2.612 \\ 
  & 64 & 0.029 & 0.024 & 0.945 & 0.887 & 2.604 \\ 
   \hline
\end{tabular}\label{tab:1}
\end{table}

\begin{table}
\caption{Simulation results under Design 2, $\sigma^{\rm eff}=2.863$.}
\centering
\begin{tabular}{cccccccc}
  \hline
$N$ & $M$ & RMSE & MAE & $95\%$ Coverage & $90\%$ Coverage & NSD \\ 
  \hline
  512 & 1 & 0.140 & 0.112 & 0.952 & 0.902 & 3.159 \\ 
  & 2 & 0.134 & 0.108 & 0.953 & 0.904 & 3.030 \\ 
  & 4 & 0.131 & 0.106 & 0.953 & 0.896 & 2.969 \\ 
  & 8 & 0.130 & 0.105 & 0.957 & 0.897 & 2.931 \\ 
  & 16 & 0.129 & 0.105 & 0.951 & 0.896 & 2.909 \\
  \hline
  1024 & 1 & 0.099 & 0.079 & 0.952 & 0.898 & 3.153 \\ 
  & 2 & 0.095 & 0.076 & 0.947 & 0.904 & 3.028 \\ 
  & 4 & 0.092 & 0.073 & 0.952 & 0.899 & 2.939 \\ 
  & 8 & 0.091 & 0.072 & 0.956 & 0.900 & 2.896 \\ 
  & 16 & 0.090 & 0.072 & 0.953 & 0.895 & 2.882 \\ 
  & 32 & 0.091 & 0.073 & 0.947 & 0.891 & 2.888 \\
  \hline
  2048 & 1 & 0.071 & 0.056 & 0.945 & 0.892 & 3.198 \\ 
  & 2 & 0.067 & 0.053 & 0.950 & 0.895 & 3.039 \\ 
  & 4 & 0.065 & 0.052 & 0.950 & 0.902 & 2.943 \\ 
  & 8 & 0.064 & 0.051 & 0.949 & 0.907 & 2.897 \\ 
  & 16 & 0.064 & 0.051 & 0.953 & 0.904 & 2.881 \\ 
  & 32 & 0.064 & 0.051 & 0.950 & 0.902 & 2.886 \\
  \hline
  4096 & 1 & 0.049 & 0.038 & 0.954 & 0.901 & 3.107 \\ 
  & 2 & 0.047 & 0.037 & 0.951 & 0.905 & 3.004 \\ 
  & 4 & 0.046 & 0.036 & 0.950 & 0.903 & 2.926 \\ 
  & 8 & 0.045 & 0.036 & 0.950 & 0.900 & 2.901 \\ 
  & 16 & 0.045 & 0.036 & 0.948 & 0.905 & 2.895 \\ 
  & 32 & 0.045 & 0.036 & 0.946 & 0.900 & 2.899 \\ 
  & 64 & 0.046 & 0.036 & 0.944 & 0.895 & 2.896 \\ 
  \hline
  8192 & 1 & 0.035 & 0.028 & 0.956 & 0.904 & 3.156 \\ 
  & 2 & 0.033 & 0.026 & 0.954 & 0.906 & 2.984 \\ 
  & 4 & 0.032 & 0.026 & 0.953 & 0.911 & 2.895 \\ 
  & 8 & 0.032 & 0.025 & 0.951 & 0.915 & 2.864 \\ 
  & 16 & 0.031 & 0.025 & 0.954 & 0.911 & 2.849 \\ 
  & 32 & 0.031 & 0.025 & 0.954 & 0.911 & 2.841 \\ 
  & 64 & 0.031 & 0.025 & 0.955 & 0.906 & 2.841 \\ 
   \hline
\end{tabular}\label{tab:2}
\end{table}

\section{Main proofs}\label{sec:main_proof}

Although in the main text we do not assume all components of $X$ to be continuous, we can analyze the problem conditional on the values of the discrete components of $X$. In each conditional case, $X$ can be seen as continuous. Therefore, in the following proofs we will only consider the case of a fully continuous $X$.

\subsection{Proof of Theorem \ref{thm:1}}
\begin{proof}[Proof of Theorem \ref{thm:1}]
We use the following proposition in this proof.
\begin{proposition}[A modified version of Proposition 2 in \cite{abadie2016matching}]\label{prop:esti}
Let  $\Lambda_{N}\left(\theta \mid \theta^{\prime}\right)$ be the difference between the values of the log-likelihood function evaluated at  $\theta$  and the value at  $\theta'$:
\begin{align*}
&\Lambda_{N}\left(\theta \mid \theta^{\prime}\right):=L\left(\theta \mid Z_{N, 1}, \ldots, Z_{N, N}\right)-L\left(\theta^{\prime} \mid Z_{N, 1}, \ldots, Z_{N, N}\right).
\end{align*}
Assume that Assumptions \ref{assumption:1}-\ref{assumption:4} hold. Assume further that $M=O(N^v)$ for some $v<1/2$ and $M\to\infty$ as $N\to\infty$. Then, under $\operatorname{P}_{\theta_{N}}$,
\begin{align*}
&\begin{array}{l}
\left(\begin{array}{c}
\sqrt{N}\left(\widehat{\tau}_{N}\left(\theta_{N}\right)-\tau\right) \\
\sqrt{N}\left(\widehat{\theta}_{N}-\theta_{N}\right) \\
\Lambda_{N}\left(\theta^{*} \mid \theta_{N}\right)
\end{array}\right) \stackrel{d}{\rightarrow} N\left(\left(\begin{array}{c}
0 \\
0 \\
-h^{\prime} I_{\theta^{*}} h / 2
\end{array}\right),\left(\begin{array}{ccc}
\sigma^{2} & c^{\prime} I_{\theta^{*}}^{-1} & -c^{\prime} h \\
I_{\theta^{*}}^{-1} c & I_{\theta^{*}}^{-1} & -h \\
-h^{\prime} c & -h^{\prime} & h^{\prime} I_{\theta^{*}} h
\end{array}\right)\right),
\end{array}&
\end{align*}
where we recall that
\begin{align*}
\ c=& \E\left[\left(\frac{\operatorname{Cov}\left(X, \mu(1, X) \mid p(X)\right)}{p(X)}+\frac{\operatorname{Cov}\left(X, \mu(0, X) \mid p(X)\right)}{1-p(X)}\right) f(X'\theta^*)\right]&
\end{align*}
and 
\begin{align*}
\sigma^{2}= & \E\left[(\bar{\mu}(1, p(X))-\bar{\mu}(0, p(X))-\tau)^{2}\right]+\E\left[\frac{\bar{\sigma}^{2}(1, p(X))}{p(X)}\right]+\E\left[\frac{\bar{\sigma}^{2}(0, p(X))}{1-p(X)}\right].&
\end{align*}
\end{proposition}

Get back to the proof. As in the proof of Lemma 1 in \cite{abadie2016matching},  Condition (ULAN) in \cite{andreou2012alternative} holds. Also, Condition (AN) in \cite{andreou2012alternative} holds due to our Proposition \ref{prop:esti} above. So, we can use Theorem 3.2 in \cite{andreou2012alternative} to obtain the result of this theorem.
\end{proof}
\begin{proof}[Proof of Proposition~\ref{prop:esti}]
The proof is based on the following three lemmas.
\begin{lemma}[Lemma 1 of \cite{abadie2016matching}]\label{lem:1}
Let  $\Delta_{N}(\theta)$  be the normalized score function, i.e., the central sequence,
\begin{align*}
\Delta_{N}(\theta) & :=\frac{1}{\sqrt{N}} \frac{\partial}{\partial \theta} L\left(\theta \mid Z_{N, 1}, \ldots, Z_{N, N}\right) \notag\\
&=\frac{1}{\sqrt{N}} \sum_{i=1}^{N} X_{i} \frac{W_{i}-p(X_i;\theta)}{p(X_i;\theta)\left(1-p(X_i;\theta)\right)} f\left(X_{i}^{\prime} \theta\right).
\end{align*}
Assume Assumptions \ref{assumption:3}(i) and (ii) hold. Then the following are true.
 \begin{enumerate}[itemsep=-.5ex,label=(\roman*)]
\item We have
\begin{align*}
&\Lambda_{N}\left(\theta^{*} \mid \theta_{N}\right) = -h^{\prime} \Delta_{N}\left(\theta_{N}\right)-\frac{1}{2} h^{\prime} I_{\theta^{*}} h+o_{\operatorname{P}_{\theta_{N}}}(1).&
\end{align*}
\item Under  $\operatorname{P}_{\theta_{N}}$, we have
\begin{align*}
&\Delta_{N}\left(\theta_{N}\right) \stackrel{d}{\rightarrow} N\left(0, I_{\theta^{*}}\right).&
\end{align*}
\item For the MLE estimator $\widehat{\theta}_N$ of $\theta_N$, we have
\begin{align*}
&\sqrt{N}\left(\widehat{\theta}_{N}-\theta_{N}\right) = I_{\theta^{*}}^{-1} \Delta_{N}\left(\theta_{N}\right)+o_{\operatorname{P}_{\theta_N}}(1).&
\end{align*}
\end{enumerate}
\end{lemma}
\begin{lemma}\label{lem:3}
Assume Assumptions \ref{assumption:1}-\ref{assumption:4} hold. Assume further that $M=O(N^v)$ for some $v<1/2$. Choose $C_1$ and $C_2$ satisfying $0<C_1<1-\bar{p}$ and $1-\underline{p}<C_2<1$. Then, under $\operatorname{P}_{\theta_{N}}$, for
\begin{align*}
    R_{N}(\theta_N):= & \frac{\ind(C_1N\leq N_0 \leq C_2N)}{\sqrt{N}} \sum_{i=1}^{N}\left(2 W_{i}-1\right)&\notag\\
&\times\left(\bar{\mu}_{\theta_N}\left(1-W_{i}, p(X_{i};\theta_N)\right)-\frac{1}{M} \sum_{j \in \mathcal{J}_{M,\theta_N}(i)} \bar{\mu}_{\theta_N}\left(1-W_{i}, p(X_{j};\theta_N)\right)\right)&\notag\\
&-\frac{1-\ind(C_1N\leq N_0 \leq C_2N)}{\sqrt{N}} \sum_{i=1}^{N}\left(\bar{\mu}_{\theta_N}\left(1, p(X_{i};\theta_N)\right)-\bar{\mu}_{\theta_N}\left(0, p(X_{i};\theta_N)\right)\right),
\end{align*}
we have
\begin{align*}
R_{N}(\theta_N)=o_{\operatorname{P}_{\theta_N}}(1).
\end{align*}
\end{lemma}
\begin{lemma} \label{lem:10}
Assume Assumptions \ref{assumption:1}-\ref{assumption:4} hold. Assume further that $M\log N/N\to0$ and $M\to\infty$ as $N\to\infty$. Choose $C_1$ and $C_2$ satisfying $0<C_1<1-\bar{p}$ and $1-\underline{p}<C_2<1$. Then, for
\begin{align*}
    D_{N}(\theta_N):= & \frac{1}{\sqrt{N}} \sum_{i=1}^{N}\left(\bar{\mu}_{\theta_N}\left(1, p(X_{i};\theta_N)\right)-\bar{\mu}_{\theta_N}\left(0, p(X_{i};\theta_N)\right)-\tau\right)& \\
& +\frac{\ind(C_1N\leq N_0 \leq C_2N)}{\sqrt{N}} \sum_{i=1}^{N}\left(2 W_{i}-1\right)\left(1+\frac{K_{M, \theta_N}(i)}{M}\right)\left(Y_{i}-\bar{\mu}_{\theta_N}\left(W_{i}, p(X_{i};\theta_N)\right)\right),&
\end{align*}
we have, under $\operatorname{P}_{\theta_{N}}$, 
\begin{align*}
&\left(\begin{array}{l}D_{N}\left(\theta_{N}\right) \\ \Delta_{N}\left(\theta_{N}\right)\end{array}\right) \stackrel{d}{\rightarrow} N\left(\left(\begin{array}{l}0 \\ 0\end{array}\right),\left(\begin{array}{cc}\sigma^{2} & c^{\prime} \\ c & I_{\theta^{*}}\end{array}\right)\right),&
\end{align*}
where we recall that
\begin{align*}
\ c=& \E\left[\left(\frac{\operatorname{Cov}\left(X, \mu(1, X) \mid p(X)\right)}{p(X)}+\frac{\operatorname{Cov}\left(X, \mu(0, X) \mid p(X)\right)}{1-p(X)}\right) f(X'\theta^*)\right]&
\end{align*}and 
\begin{align*}
\sigma^{2}= & \E\left[(\bar{\mu}(1, p(X))-\bar{\mu}(0, p(X))-\tau)^{2}\right]+\E\left[\frac{\bar{\sigma}^{2}(1, p(X))}{p(X)}\right]+\E\left[\frac{\bar{\sigma}^{2}(0, p(X))}{1-p(X)}\right].&
\end{align*}
\end{lemma}

Get back to the proof. Lemma \ref{lem:1} implies that, in order to prove Proposition \ref{prop:esti},  it suffices to show that
\begin{align*}
&\begin{array}{l}
\left(\begin{array}{c}
\sqrt{N}\left(\widehat{\tau}_{N}\left(\theta_{N}\right)-\tau\right) \\
\Delta_N(\theta_{N})
\end{array}\right) \stackrel{d}{\rightarrow} N\left(\left(\begin{array}{c}
0 \\
0 
\end{array}\right),\left(\begin{array}{cc}
\sigma^{2} & c' \\
c & I_{\theta^{*}} 
\end{array}\right)\right)
\end{array}&
\end{align*}
holds under $\operatorname{P}_{\theta_N}$. For any $C_1$ and $C_2$ satisfying $0<C_1<1-\bar{p}$ and $1-\underline{p}<C_2<1$, $C_1N\leq N_0\leq C_2N$ almost surely holds under $\operatorname{P}_{\theta_N}$ as $N\to\infty$, so we can make the following decomposition,
\begin{align*}
\sqrt{N}\left(\widehat{\tau}_{N}\left(\theta_{N}\right)-\tau\right)=D_N(\theta_N)+R_N(\theta_N)+o_{\operatorname{P}_{\theta_N}}(1).
\end{align*}
Then, combining Lemma \ref{lem:3} and Lemma \ref{lem:10} yields the proposition.
\end{proof}
\subsection{Proof of Theorem \ref{thm:var}}
\begin{proof}[Proof of Theorem~\ref{thm:var}]
As a matter of fact, we only have to handle the consistency of $\widehat{\sigma}^2$ since the analyses of $\widehat{c}$ and $\widehat{I}_{\theta^*}$ are the same as the ones in \cite{abadie2016matching}. 

{\bf Part I.} According to Proposition \ref{prop:esti} and Le Cam's third lemma, we know that 
\begin{align*}
    \widehat{\theta}_N-\theta^*\overset{d}{\to} N(0,I_{\theta^*}^{-1})
\end{align*}
holds under $\operatorname{P}_{\theta^*}$. So, for any $\epsilon>0$, there exists a positive integer $H>0$ such that for all sufficiently large $N$,
\begin{align*}
\operatorname{P}\left(\sqrt{N}\|\widehat{\theta}_N-\theta^*\|\geq H\right)<\frac{1}{2}\epsilon
\end{align*}
always holds. Consider {\it all} real vector sequences $\theta_{(N)}$ satisfying  $\sqrt{N}\|\theta_{(N)}-\theta^*\|<H$.  Let $\tilde{\sigma}^2$ be the statistic obtained by replacing all $\hat{\theta}_N$ in $\widehat{\sigma}^2$ with $\theta_{(N)}$. If $\tilde{\sigma}^2\overset{\P}{\to}\sigma^2$ {\it uniformly} for any choice of $\theta_{(N)}$ mentioned above, it follows that $\widehat{\sigma}^2\overset{\P}{\to}\sigma^2$.

It then remains to prove that $\tilde{\sigma}^2\overset{\P}{\to}\sigma^2$ uniformly. As shown in the proof of Lemma 1 in \cite{abadie2016matching}, the conditions of Proposition 2.1.2 in \cite{bickel1993efficient} are satisfied in our problem. Since $\sqrt{N}\|\theta_{(N)}-\theta^*\|<H$ holds for all $\theta_{(N)}$, the uniform convergence in that proposition yields
\begin{align*}
\Lambda_{N}\left(\theta^{*} \mid \theta_{(N)}\right) &= -\sqrt{N}(\theta_{(N)}-\theta^*)^{\prime} \Delta_{N}\left(\theta^*\right)-\frac{1}{2} (\theta_{(N)}-\theta^*)^{\prime} I_{\theta^{*}} (\theta_{(N)}-\theta^*)+o_{\operatorname{P}_{\theta_{(N)}}}(1)&\\
&=O_{\operatorname{P}_{\theta_{(N)}}}(1),
\end{align*}
uniformly in $\theta_{(N)}$. For any statistics $T_N$ of $(Z_{1},\cdots,Z_{N})$ and any Borel set $\mathcal{B}$, we have
\begin{align*}
    \operatorname{P}(T_N\in\mathcal{B})=\E_{\theta_{(N)}}\left[\Lambda_{N}\left(\theta^{*} \mid \theta_{(N)}\right)\ind(T_N\in\mathcal{B})\right].
\end{align*}
Accordingly, we know that the uniform convergence in probability of $\tilde{\sigma}^2$ under $\operatorname{P}_{\theta_{(N)}}$ leads to the uniform convergence in probability of $\tilde{\sigma}^2$ under $\operatorname{P}$.  We thus only have to prove that 
\[
\tilde{\sigma}^2\overset{\operatorname{P}_{\theta_{(N)}}}{\to}\sigma^2 \text{ holds uniformly.}
\]

{\bf Part II.} Without loss of generality, assume that for any $N$ and the $\epsilon$ in Assumption \ref{assumption:3}(iv), $\|\theta_{(N)}-\theta^*\|<\epsilon$ holds. Let 
\[
\epsilon_i:=Y_i-\bar{\mu}_{\theta_{(N)}}(W_i,p(X_i;\theta_{(N)})), 
\]
one can make the following decomposition:
\begin{align*}
    \tilde{\sigma}^2=&\frac{1}{N} \sum_{i=1}^{N}\left(Y_i-\frac{1}{M} \sum_{j \in \mathcal{J}_{M}(i, \theta_{(N)})} Y_j\right)^2-\widehat{\tau}_{N}^2(\theta_{(N)})\notag\\
&+\frac{1}{N} \sum_{i=1}^{N}\left(\left(\frac{K_{M, \theta_{(N)}}(i)}{M}\right)^{2}+\frac{2M-1}{M}\left(\frac{K_{M, \theta_{(N)}}(i)}{M}\right)\right) \widehat{\bar{\sigma}}_{\theta_{(N)}}^{2}\left(W_i, p(X_i)\right)\notag\\
=&-\frac{1}{N} \sum_{i=1}^{N}\left(\bar{\mu}_{\theta_{(N)}}(1-W_i,p(X_i;\theta_{(N)}))-\frac{1}{M} \sum_{j \in \mathcal{J}_{M}(i, \theta_{(N)})} \bar{\mu}_{\theta_{(N)}}(1-W_i,p(X_j;\theta_{(N)}))\right)^2\\
&+\frac{2}{N}\sum_{i=1}^{N}\left(Y_i-\frac{1}{M} \sum_{j \in \mathcal{J}_{M}(i, \theta_{(N)})} Y_j\right)\notag\\
&\times\left(\bar{\mu}_{\theta_{(N)}}(1-W_i,p(X_i;\theta_{(N)}))-\frac{1}{M} \sum_{j \in \mathcal{J}_{M}(i, \theta_{(N)})} \bar{\mu}_{\theta_{(N)}}(1-W_i,p(X_j;\theta_{(N)}))\right)\\
&+\frac{1}{N} \sum_{i=1}^{N}\left(\bar{\mu}_{\theta_{(N)}}(W_i,p(X_i;\theta_{(N)}))-\bar{\mu}_{\theta_{(N)}}(1-W_i,p(X_i;\theta_{(N)}))\right)^2-\widehat{\tau}_{N}^2(\theta_{(N)})\\
&+\frac{1}{N} \sum_{i=1}^{N}\left(\epsilon_i-\frac{1}{M} \sum_{j \in \mathcal{J}_{M}(i, \theta_{(N)})}\epsilon_j\right)^2\\
&+\frac{2}{N}\sum_{i=1}^{N}\left(\epsilon_i-\frac{1}{M} \sum_{j \in \mathcal{J}_{M}(i, \theta_{(N)})}\epsilon_j\right)\left(\bar{\mu}_{\theta_{(N)}}(W_i,p(X_i;\theta_{(N)}))-\bar{\mu}_{\theta_{(N)}}(1-W_i,p(X_i;\theta_{(N)}))\right)\\
&+\frac{1}{N} \sum_{i=1}^{N}\left(\left(\frac{K_{M, \theta_{(N)}}(i)}{M}\right)^{2}+\frac{2M-1}{M}\left(\frac{K_{M, \theta_{(N)}}(i)}{M}\right)\right) \left(\widehat{\bar{\sigma}}^{2}\left(W_i, p(X_i)\right)-\bar{\sigma}_{\theta_{(N)}}^{2}\left(W_i, p(X_i;\theta_{(N)})\right)\right)\\
&+\frac{1}{N} \sum_{i=1}^{N}\left(\left(\frac{K_{M, \theta_{(N)}}(i)}{M}\right)^{2}+\frac{2M-1}{M}\left(\frac{K_{M, \theta_{(N)}}(i)}{M}\right)\right) \bar{\sigma}_{\theta_{(N)}}^{2}\left(W_i, p(X_i;\theta_{(N)})\right)\\
&=: I_1+I_2+I_3+I_4+I_5+I_6+I_7.
\end{align*}

The following proposition controls the terms $I_1$ to $I_7$. 
\begin{proposition}\label{prop:var1}
    Assume Assumptions \ref{assumption:1}-\ref{assumption:4} hold. For a family of $\theta_{(N)}$ s.t. $\theta_{(N)}\to\theta^*$ uniformly, there are uniform convergences as follows.
    \begin{align*}
    &I_i\overset{\operatorname{P}_{\theta_{(N)}}}{\to}0 \ for \ i=1,2,5,6,\\ &I_3\overset{\operatorname{P}_{\theta_{(N)}}}{\to}\E\left[(\bar{\mu}(1,p(X))-\bar{\mu}(0,p(X))-\tau)^2\right],\\ &I_4\overset{\operatorname{P}_{\theta_{(N)}}}{\to}\E\left[\bar{\sigma}^2\left(W, p(X)\right)\right],\\ & I_7\overset{\operatorname{P}_{\theta_{(N)}}}{\to}\E\left[\frac{\bar{\sigma}^2\left(1, p(X)\right)}{p(X)}+\frac{\bar{\sigma}^2\left(0, p(X)\right)}{1-p(X)}\right]-\E\left[\bar{\sigma}^2\left(W, p(X)\right)\right]. 
    \end{align*}
\end{proposition}
With this proposition, the proof is then finished.
\end{proof}
\begin{proof}[Proof of Proposition \ref{prop:var1}]
The proof uses Lemma \ref{lem:3}, Lemma \ref{lem:10}, and the following two lemmas.
\begin{lemma}[An enhanced version of Lemma S.8 in \cite{abadie2016matching}]\label{lem:7}
Assume Assumptions \ref{assumption:1}-\ref{assumption:4} hold. Choose $C_1$ and $C_2$ satisfying $0<C_1<1-\bar{p}$ and $1-\underline{p}<C_2<1$. Then, for any positive integer $k$, 
\[
\ind(C_1N\leq N_0\leq C_2N)\E_{\theta_N}\left[\left(\frac{K_{M,\theta_N}(i)}{M}\right)^{k} \mid W_i=w, N_0\right] 
\]
is uniformly bounded in $N$, $N_0$, $M\leq\min\{N_0,N_1\}$, and $w=0,1$.
\end{lemma}
\begin{lemma}[A modified version of Lemma S.10 in \cite{abadie2016matching}]\label{lem:9}
Assume Assumptions \ref{assumption:1}-\ref{assumption:4} hold. Assume further that $M=O(N^v)$ for some $v<1/2$ and $M\to\infty$ as $N\to\infty$. Choose $C_1$ and $C_2$ satisfying $0<C_1<1-\bar{p}$ and $1-\underline{p}<C_2<1$. Under $\operatorname{P}_{\theta_N}$, for any $w=0,1$ and any uniformly bounded and equicontinuous functions $\bar{m}(w,p;\theta_N)$, which is simplified as $\bar{m}(w,p)$, we have
\begin{align*}
&\frac{\ind(C_1N\leq N_0 \leq C_2N)}{N_{w}} \sum_{i: W_i=w} \bar{m}(w,p(X_i;\theta_N)) \frac{K_{M,\theta_N}(i)}{M}\\
\stackrel{\operatorname{P}_{\theta_N}}{\rightarrow}~~&\E_{\theta_N}\left[\bar{m}(w,p(X;\theta_N))\left(\frac{p(X;\theta_N)}{1-p(X;\theta_N)}\right)^{1-2 w} \mid W=w\right]
\end{align*}
and
\begin{align*}
&\frac{\ind(C_1N\leq N_0 \leq C_2N)}{N_{w}} \sum_{i: W_i=w} \bar{m}(w,p(X_i;\theta_N)) \left(\frac{K_{M,\theta_N}(i)}{M}\right)^{2}\\
\stackrel{\operatorname{P}_{\theta_N}}{\rightarrow} ~~&\E_{\theta_N}\left[\bar{m}(w,p(X;\theta_N))\left(\frac{p(X;\theta_N)}{1-p(X;\theta_N)}\right)^{2(1-2 w)} \mid W=w\right].
\end{align*}
\end{lemma}
Notably, the above lemmas can be established in a uniform fashion. Later, they will also be used to prove Theorem \ref{thm:1}, where we focus on the case of a particular $\theta_N$. However, these lemmas still hold when we change $\theta_N$ to any $\theta_{(N)}$ converging to $\theta^*$. Furthermore, all the convergences and the uniform boundedness they achieve are uniform in a family of $\theta_{(N)}$ as long as $\theta_{(N)}\to\theta^*$ uniformly. We do not present the lemmas with a family of $\theta_{(N)}$ for notation symplicity.

Get back to the proof. As for $I_1$, since the convergence in Lemma \ref{lem:3} holds uniformly for any sequences of $\theta_{(N)}$ uniformly converging to $\theta^*$, and $\bar{\mu}_{\theta_{(N)}}(w,p)$ is uniformly bounded in $\theta_{(N)}$, $w$, and $p$, we know that 
\[
I_1\overset{\operatorname{P}_{\theta_{(N)}}}{\to}0 \text{ uniformly in }\theta_{(N)}.
\] 

For $I_2$, let 
\[
\mathbf{W}_N:=(W_{1},\cdots,W_{N})~~ {\rm and}~~ \mathbf{F}_N:=(X_{1}'\theta_{(N)},\cdots,X_{N}'\theta_{(N)}). 
\]
We then have
\begin{align*}
    &\E_{\theta_{(N)}}\left[\ind(C_1N\leq N_0\leq C_2N)I_2^2\mid \mathbf{W}_N,\mathbf{F}_N\right]\\
    \leq& \frac{\ind(C_1N\leq N_0\leq C_2N)16C_Y^4}{N^2}\\
    &\times\left(\sum_{i=1}^{N}\left|\bar{\mu}_{\theta_{(N)}}(1-W_i,p(X_i;\theta_{(N)}))-\frac{1}{M} \sum_{j \in \mathcal{J}_{M}(i, \theta_{(N)})} \bar{\mu}_{\theta_{(N)}}(1-W_i,p(X_j;\theta_{(N)}))\right|\right)^2,
\end{align*}
where $C_Y$ is the upper bound of $\left(\E_{\theta_{(N)}}\left[|Y|^4|W,p(X;\theta_{(N)}))\right]\right)^{\frac{1}{4}}$. By the uniform convergence version of Lemma \ref{lem:3}, we know that the conditional expectation is $o_{\operatorname{P}_{\theta_{(N)}}}(1)$ uniformly in $\theta_{(N)}$. Then it is straightforward to show that
\[
\ind(C_1N\leq N_0\leq C_2N)I_2\overset{\operatorname{P}_{\theta_{(N)}}}{\to}0, 
\]
that is, 
\[
I_2\overset{\operatorname{P}_{\theta_{(N)}}}{\to}0 \text{ uniformly in }\theta_{(N)}.
\] 

As for $I_3$, since the two terms are both asymptotically normal and their asymptotic variances are uniformly bounded in $\theta_{(N)}$, we have the uniform convergence, that is, 
\begin{align*}
I_3&\overset{\operatorname{P}_{\theta_{(N)}}}{\to}\E_{\theta_{(N)}}\left[(\bar{\mu}_{\theta_{(N)}}(1,p(X;\theta_{(N)}))-\bar{\mu}_{\theta_{(N)}}(0,p(X;\theta_{(N)})))^2\right]-\tau^2\\
&\to\E\left[(\bar{\mu}(1,p(X))-\bar{\mu}(0,p(X))-\tau)^2\right].
\end{align*}

As for $I_4$, we decompose
\begin{align}\label{I4_decom}
    I_4=\frac{1}{N} \sum_{i=1}^{N}\epsilon_i^2\left(\frac{K_{M, \theta_{(N)}}(i)}{M^2}+1\right)-\frac{2}{NM} \sum_{i=1}^{N}\sum_{j \in \mathcal{J}_{M}(i, \theta_{(N)})}\epsilon_i\epsilon_j+\frac{1}{NM^2} \sum_{i=1}^{N}\sum_{j, k \in \mathcal{J}_{M}(i, \theta_{(N)}),j_1\neq j_2}\epsilon_k\epsilon_j.
\end{align}    
Similar to the calculation of $I_3$, we have 
\begin{align*}
\frac{1}{N} \sum_{i=1}^{N}\epsilon_i^2\overset{\operatorname{P}_{\theta_{(N)}}}{\to}\E_{\theta_{(N)}}\left[\bar{\sigma}_{\theta_{(N)}}^{2}(W, p(X;\theta_{(N)}))\right]\to\E\left[\bar{\sigma}^{2}\left(W, p(X)\right)\right]
\end{align*}
holds uniformly in $\theta_{(N)}$. For the rest terms on the right-hand side of \eqref{I4_decom}, we have 
\begin{align*}
    &\E_{\theta_{(N)}}\left[\frac{\ind(C_1N\leq N_0\leq C_2N)}{N} \sum_{i=1}^{N}\epsilon_i^2\frac{K_{M, \theta_{(N)}}(i)}{M^2}|\mathbf{W}_N,\mathbf{F}_N\right]\\
    \leq &\frac{\ind(C_1N\leq N_0\leq C_2N)C_Y^2}{NM} \sum_{i=1}^{N}\left(\frac{K_{M, \theta_{(N)}}(i)}{M}\right),\\ 
    &\E_{\theta_{(N)}}\left[\left(\frac{\ind(C_1N\leq N_0\leq C_2N)2}{NM} \sum_{i=1}^{N}\sum_{j \in \mathcal{J}_{M}(i, \theta_{(N)})}\epsilon_i\epsilon_j\right)^2\mid\mathbf{W}_N,\mathbf{F}_N\right]\\
    \leq&\frac{\ind(C_1N\leq N_0\leq C_2N)4}{N^2M^2} \sum_{i=1}^{N}\sum_{j \in \mathcal{J}_{M}(i, \theta_{(N)})}2\bar{\sigma}_{\theta_{(N)}}^{2}(W_i, p(X_i;\theta_{(N)}))\bar{\sigma}_{\theta_{(N)}}^{2}(W_j, p(X_j;\theta_{(N)}))\\
    \leq&\frac{\ind(C_1N\leq N_0\leq C_2N)8C_Y^4}{NM},\\
    &\E_{\theta_{(N)}}\left[\left(\frac{\ind(C_1N\leq N_0\leq C_2N)}{NM^2} \sum_{i=1}^{N}\sum_{j, k \in \mathcal{J}_{M}(i, \theta_{(N)}),j_1\neq j_2}\epsilon_k\epsilon_j\right)^2\mid\mathbf{W}_N,\mathbf{F}_N\right]\\
    \leq&\frac{\ind(C_1N\leq N_0\leq C_2N)C_Y^4}{N^2M^2} \sum_{i=1}^{N}\left(\frac{2K_{M, \theta_{(N)}}(i)}{M}\right)^2,
\end{align*}
where the inequalities about the last two terms are done by bounding the maximal number of the non-zero terms in the full expansion of the two expectations and bounding all terms in the form of $\bar{\sigma}_{\theta_{(N)}}^{2}(W_i, p(X_i;\theta_{(N)}))$ by $C_Y^2$. According to Lemma \ref{lem:7}, all the three terms are $o_{\operatorname{P}_{\theta_{(N)}}}(1)$ uniformly in $\theta_{(N)}$, and accordingly 
\[
I_4\overset{\operatorname{P}_{\theta_{(N)}}}{\to}\E\left[\bar{\sigma}^{2}\left(W, p(X)\right)\right] \text{ uniformly in }\theta_{(N)}.
\] 

As for $I_5$, one can derive that
\begin{align*}
    &\E_{\theta_{(N)}}\left[\ind(C_1N\leq N_0\leq C_2N)I_5^2\mid\mathbf{W}_N,\mathbf{F}_N\right]\\
    \leq&\frac{\ind(C_1N\leq N_0\leq C_2N)4}{N^2}\sum_{i=1}^{N}\E_{\theta_{(N)}}\left[\left(\left(1+\frac{K_{M, \theta_{(N)}}(i)}{M}\right)2C_Y\right)^2\epsilon_i^2\mid\mathbf{W}_N,\mathbf{F}_N\right]\\
    \leq &\frac{\ind(C_1N\leq N_0\leq C_2N)16C_Y^4}{N^2}\sum_{i=1}^{N}\left(1+\frac{K_{M, \theta_{(N)}}(i)}{M}\right)^2
\end{align*}
holds and the conditional expectation is $o_{\operatorname{P}_{\theta_{(N)}}}(1)$ uniformly in $\theta_{(N)}$. So $I_5\overset{\operatorname{P}_{\theta_{(N)}}}{\to}0$ uniformly in $\theta_{(N)}$. 

As for $I_6$, we can decompose it as $I_6=I_{6,1}+I_{6,2}+o_{\operatorname{P}_{\theta_{(N)}}}(1)$ uniformly in $\theta_{(N)}$, where
\begin{align*}
    I_{6,1}:=&\frac{1}{N} \sum_{i=1}^{N}\left(\left(\frac{K_{M, \theta_{(N)}}(i)}{M}\right)^{2}+\frac{2M-1}{M}\left(\frac{K_{M, \theta_{(N)}}(i)}{M}\right)\right)\\
    &\times\left(E_{\theta_{(N)}}\left[\ind(C_1N\leq N_0\leq C_2N)\widehat{\bar{\sigma}}^{2}\left(W_i, p(X_i)\right)\mid \mathbf{F}_N,\mathbf{W}_N\right]-\bar{\sigma}_{\theta_{(N)}}^{2}\left(W_i, p(X_i;\theta_{(N)})\right)\right),\\
    I_{6,2}:=&\frac{1}{N} \sum_{i=1}^{N}\left(\left(\frac{K_{M, \theta_{(N)}}(i)}{M}\right)^{2}+\frac{2M-1}{M}\left(\frac{K_{M, \theta_{(N)}}(i)}{M}\right)\right)\\
    &\times\left(\ind(C_1N\leq N_0\leq C_2N)\widehat{\bar{\sigma}}^{2}\left(W_i, p(X_i)\right)\right.\\
    &\left.-E_{\theta_{(N)}}\left[\ind(C_1N\leq N_0\leq C_2N)\widehat{\bar{\sigma}}^{2}\left(W_i, p(X_i)\right)\mid \mathbf{F}_N,\mathbf{W}_N\right]\right).
\end{align*}
For $I_{6,1}$, we have
\begin{align*}
    &\E_{\theta_{(N)}}\left[\ind(C_1N\leq N_0\leq C_2N)\widehat{\bar{\sigma}}^{2}\left(W_i, p(X_i)\right)\mid \mathbf{F}_N,\mathbf{W}_N\right]-\bar{\sigma}_{\theta_{(N)}}^{2}\left(W_i, p(X_i;\theta_{(N)})\right)\\
    =&\E_{\theta_{(N)}}\left[\frac{\ind(C_1N\leq N_0\leq C_2N)}{Q-1}\sum_{j \in \mathcal{H}_{Q}\left(i, \bar{\theta}_{N}\right)}\left(Y_j-\frac{1}{Q} \sum_{k \in \mathcal{H}_{Q}\left(i, \bar{\theta}_{N}\right)} Y_k\right)^{2}\mid \mathbf{F}_N,\mathbf{W}_N\right]\\
    &-\bar{\sigma}_{\theta_{(N)}}^{2}\left(W_i, p(X_i;\theta_{(N)})\right)\\
    =&\frac{\ind(C_1N\leq N_0\leq C_2N)}{Q-1} \sum_{j \in \mathcal{H}_{Q}\left(i, \bar{\theta}_{N}\right)}\left(\bar{\mu}_{\theta_{(N)}}(W_j,p(X_j;\theta_{(N)}))-\frac{1}{Q} \sum_{k \in \mathcal{H}_{Q}\left(i, \bar{\theta}_{N}\right)} \bar{\mu}_{\theta_{(N)}}(W_k,p(X_k;\theta_{(N)}))\right)^{2}\\
    &+\frac{\ind(C_1N\leq N_0\leq C_2N)}{Q}\sum_{j \in \mathcal{H}_{Q}\left(i, \bar{\theta}_{N}\right)}\bar{\sigma}^2_{\theta_{(N)}}(W_j,p(X_j;\theta_{(N)}))-\bar{\sigma}_{\theta_{(N)}}^{2}\left(W_i, p(X_i;\theta_{(N)})\right).
\end{align*}
By the derivations of Lemma \ref{lem:2}, we can analogously obtain that
\begin{align}\label{eq:use_l2}
    \max_{j,k\in \mathcal{H}_{Q}\left(i, \bar{\theta}_{N},\right),\forall i}\left|p(X_j;\theta_{(N)})-p(X_k;\theta_{(N)})\right|=o_{\operatorname{P}_{\theta_{(N)}}}(1)
\end{align}
uniformly in $\theta_{(N)}$. Then, with the Lipschitz property and the equicontinuity in Assumption \ref{assumption:3}(iv), we have the uniform convergences that
\begin{align*}
    &\max_{j,k\in \mathcal{H}_{Q}\left(i, \bar{\theta}_{N},\right),\forall i}\left|\bar{\mu}_{\theta_{(N)}}(W_j,p(X_j;\theta_{(N)}))-\bar{\mu}_{\theta_{(N)}}(W_k,p(X_k;\theta_{(N)}))\right|=o_{\operatorname{P}_{\theta_{(N)}}}(1),\\
    &\max_{j,k\in \mathcal{H}_{Q}\left(i, \bar{\theta}_{N},\right),\forall i}\left|\bar{\sigma}^2_{\theta_{(N)}}(W_j,p(X_j;\theta_{(N)}))- \bar{\sigma}^2_{\theta_{(N)}}(W_k,p(X_k;\theta_{(N)}))\right|=o_{\operatorname{P}_{\theta_{(N)}}}(1).
\end{align*}
Accordingly,
\begin{align*}
    \max_{i}\left|\E_{\theta_{(N)}}\left[\ind(C_1N\leq N_0\leq C_2N)\widehat{\bar{\sigma}}^{2}\left(W_i, p(X_i)\right)\mid \mathbf{F}_N,\mathbf{W}_N\right]-\bar{\sigma}_{\theta_{(N)}}^{2}\left(W_i, p(X_i;\theta_{(N)})\right)\right|=o_{\operatorname{P}_{\theta_{(N)}}}(1)
\end{align*}
uniformly in $\theta_{(N)}$. Since Lemma \ref{lem:7} yields that
\begin{align}\label{I_6_mid}
    \frac{\ind(C_1N\leq N_0\leq C_2N)}{N} \sum_{i=1}^{N}\left(\left(\frac{K_{M, \theta_{(N)}}(i)}{M}\right)^{2}+\frac{2M-1}{M}\left(\frac{K_{M, \theta_{(N)}}(i)}{M}\right)\right)=O_{\operatorname{P}_{\theta_{(N)}}}(1)
\end{align}
uniformly in $\theta_{(N)}$, we know that $I_{6,1}\overset{\operatorname{P}_{\theta_{(N)}}}{\to}0$ uniformly in $\theta_{(N)}$. 

For $I_{6,2}$, we have
\begin{align*}
    &\ind(C_1N\leq N_0\leq C_2N)\widehat{\bar{\sigma}}^{2}\left(W_i, p(X_i)\right)-\E_{\theta_{(N)}}\left[\ind(C_1N\leq N_0\leq C_2N)\widehat{\bar{\sigma}}^{2}\left(W_i, p(X_i)\right)\mid \mathbf{F}_N,\mathbf{W}_N\right]\\
    =&\frac{\ind(C_1N\leq N_0\leq C_2N)}{Q}\sum_{j\in \mathcal{H}_{Q}\left(i, \bar{\theta}_{N}\right)}\left(Y_j^2-\E_{\theta_{(N)}}\left[Y_j^2\mid W_j,p(X_j;\theta_{(N)})\right]\right)\\
    &-\frac{\ind(C_1N\leq N_0\leq C_2N)}{Q(Q-1)}\sum_{j,k\in\mathcal{H}_{Q}\left(i, \bar{\theta}_{N}\right),j\neq k}\epsilon_j\epsilon_k.
\end{align*}
Using the facts that
\begin{align*}
    &\E_{\theta_{(N)}}\left[\left(\frac{\ind(C_1N\leq N_0\leq C_2N)}{Q}\sum_{j\in \mathcal{H}_{Q}\left(i, \bar{\theta}_{N}\right)}\left(Y_j^2-\E_{\theta_{(N)}}\left[Y_j^2\mid W_j,p(X_j;\theta_{(N)})\right]\right)\right)^2\mid\mathbf{F}_N,\mathbf{W}_N\right]\leq \frac{C_Y^4}{Q},\\
    &\E_{\theta_{(N)}}\left[\left(\frac{\ind(C_1N\leq N_0\leq C_2N)}{Q(Q-1)}\sum_{j,k\in\mathcal{H}_{Q}\left(i, \bar{\theta}_{N}\right),j\neq k}\epsilon_j\epsilon_k\right)^2\mid\mathbf{F}_N,\mathbf{W}_N\right]\leq \frac{C_Y^44Q(Q-1)/2}{Q^2(Q-1)^2},
\end{align*}
we know that 
\begin{align*}
    &\E_{\theta_{(N)}}\left[\left(\ind(C_1N\leq N_0\leq C_2N)\widehat{\bar{\sigma}}^{2}\left(W_i, p(X_i)\right)\right.\right.\\
    &\left.\left.~~~~~~-E_{\theta_{(N)}}\left[\ind(C_1N\leq N_0\leq C_2N)\widehat{\bar{\sigma}}^{2}\left(W_i, p(X_i)\right)\mid \mathbf{F}_N,\mathbf{W}_N\right]\right)^2\mid\mathbf{F}_N,\mathbf{W}_N\right]=o(1)   
\end{align*}
uniformly in $\mathbf{F}_N$ and $\mathbf{W}_N$ and $\theta_{(N)}$. Since
\begin{align*}
    &\E_{\theta_{(N)}}\left[I_6^2\mid\mathbf{F}_N,\mathbf{W}_N\right]\\
    \leq&\frac{1}{N} \sum_{i=1}^{N}\E_{\theta_{(N)}}\left[\left(\left(\frac{K_{M, \theta_{(N)}}(i)}{M}\right)^{2}+\frac{2M-1}{M}\left(\frac{K_{M, \theta_{(N)}}(i)}{M}\right)\right)^2\right.\\
    &\left.\times\left(\ind(C_1N\leq N_0\leq C_2N)\widehat{\bar{\sigma}}^{2}\left(W_i, p(X_i)\right)\right.\right.\\
    &\left.\left.-E_{\theta_{(N)}}\left[\ind(C_1N\leq N_0\leq C_2N)\widehat{\bar{\sigma}}^{2}\left(W_i, p(X_i)\right)\mid \mathbf{F}_N,\mathbf{W}_N\right]\right)^2\mid\mathbf{F}_N,\mathbf{W}_N\right]\\
    \leq&\frac{1}{N} \sum_{i=1}^{N}\left(\left(\frac{K_{M, \theta_{(N)}}(i)}{M}\right)^{2}+\frac{2M-1}{M}\left(\frac{K_{M, \theta_{(N)}}(i)}{M}\right)\right)^2\\
    &\E_{\theta_{(N)}}\left[\left(\ind(C_1N\leq N_0\leq C_2N)\widehat{\bar{\sigma}}^{2}\left(W_i, p(X_i)\right)\right.\right.\\
    &\left.\left.-E_{\theta_{(N)}}\left[\ind(C_1N\leq N_0\leq C_2N)\widehat{\bar{\sigma}}^{2}\left(W_i, p(X_i)\right)\mid \mathbf{F}_N,\mathbf{W}_N\right]\right)^2\mid\mathbf{F}_N,\mathbf{W}_N\right],
\end{align*}
we can use (\ref{I_6_mid}) to derive that $\E_{\theta_{(N)}}\left[I_6^2\mid\mathbf{F}_N,\mathbf{W}_N\right]=o_{\operatorname{P}_{\theta_{(N)}}}(1)$ uniformly in $\theta_{(N)}$. That is to say, $I_{6,2}\overset{\operatorname{P}_{\theta_{(N)}}}{\to}0$ and $I_{6}\overset{\operatorname{P}_{\theta_{(N)}}}{\to}0$ uniformly in $\theta_{(N)}$. 

As for $I_7$, by Lemma \ref{lem:9} and the derivations in Lemma \ref{lem:10}, 
\begin{align*}
I_7\overset{\operatorname{P}_{\theta_{(N)}}}{\to}&\E_{\theta_{(N)}}\left[\bar{\sigma}_{\theta_{(N)}}^{2}\left(1, p(X;\theta_{(N)})\right)\left(\left(\frac{1-p(X;\theta_{(N)})}{p(X;\theta_{(N)})}+1\right)^2-1\right)W\right]\\
&+\E_{\theta_{(N)}}\left[\bar{\sigma}_{\theta_{(N)}}^{2}\left(0, p(X;\theta_{(N)})\right)\left(\left(\frac{p(X;\theta_{(N)})}{1-p(X;\theta_{(N)})}+1\right)^2-1\right)(1-W)\right]\\
\to& \E\left[\bar{\sigma}^{2}\left(1, p(X)\right)\left(\frac{1}{p(X)}-p(X)\right)+\bar{\sigma}^{2}\left(0, p(X)\right)\left(\frac{1}{1-p(X)}-(1-p(X))\right)\right]\\
=& \E\left[\frac{\bar{\sigma}^2\left(1, p(X)\right)}{p(X)}+\frac{\bar{\sigma}^2\left(0, p(X)\right)}{1-p(X)}\right]-\E\left[\bar{\sigma}^2\left(W, p(X)\right)\right].
\end{align*}
Combining them all finishes the proof.
\end{proof}

\section{Proofs of the lemmas in Section \ref{sec:main_proof}}\label{sec:proof_direct}
This section proves the lemmas used in Section \ref{sec:main_proof}. In order to avoid circular arguments, we put the proof of Lemma \ref{lem:10}, which uses Lemma \ref{lem:7} and Lemma \ref{lem:9}, at last. We will further introduce some new lemmas in this section, whose proofs will be put in the next section. 

\begin{proof}[Proof of Lemma \ref{lem:1}]
    It's exactly Lemma 1 in 
 \cite{abadie2016matching}.
\end{proof}

\begin{proof}[Proof of Lemma \ref{lem:3}]
To prove this lemma, we use the following lemma.
\begin{lemma}[An enhanced version of Lemma S.1 in \cite{abadie2016matching}]\label{lem:2}
Let  $F_{0}$  and  $F_{1}$ be two continuous distribution functions with a common real interval support with the densities $f_0$ and $f_1$, respectively. Assume the density ratio $f_{1}(x) / f_{0}(x) \leq \bar{r}<\infty$ for any $x$ in the support of $F_0, F_1$. Suppose
\[
\tilde{X}_{0,1}, \ldots, \tilde{X}_{0, n_{0}} \stackrel{\rm i.i.d.}{\sim}  F_{0}  ~~\text{ and }~~ \tilde{X}_{1,1}, \ldots, \tilde{X}_{1, n_{1}} \stackrel{\rm i.i.d.}{\sim}  F_{1}
\]
and denote $n=n_{0}+n_{1}$.  For  $1 \leq i \leq n_{1}$ and  $1 \leq m \leq M \leq n_{0}$, let  $\left|U_{n_{0}, n_{1}, i}\right|_{(m)}$ be the $m$-th order statistic of  
\[
\left\{\left|\tilde{X}_{1, i}-\tilde{X}_{0,1}\right|, \ldots,\left|\tilde{X}_{1, i}-\tilde{X}_{0, n_{0}}\right|\right\}. 
\]
Then, as long as $n_0\geq4$, we have (i)
\begin{align*}
&\operatorname{P}\left(\max_{i}\Big|U_{n_{0}, n_{1}, i}\Big|_{(M)}>2\max_{i}\left|F_0^{-1}\left(\frac{i+1}{\lfloor n_0^{1/2}\log n_0\rfloor}\right)-F_0^{-1}\left(\frac{i}{\lfloor n_0^{1/2}\log n_0\rfloor}\right)\right|\right)\notag\\
&\leq \exp\left(\log M+1+(M-1) \log \left(\frac{n_0e}{(\lfloor n_0^{1/2}\log n_0\rfloor-1)M}\right)-\frac{n_0}{\lfloor n_0^{1/2}\log n_0\rfloor}\right).&
\end{align*}
If further assuming that the support of  $F_{0}$  and  $F_{1}$  is an interval inside $[0,1]$, we have (ii)
\begin{align*}
&\E\left[\frac{1}{\sqrt{N}} \sum_{i=1}^{n_{1}} \frac{1}{M} \sum_{m=1}^{M}\left|U_{n_{0}, n_{1}, i}\right|_{(m)}\right]\leq \frac{3\bar{r}n_{1}}{n^{1 / 2}\lfloor n_0^{1/2}\log n_0\rfloor}&\notag\\
&~~~~~~+\exp\left(\log \frac{Mn_1}{n^{1/2}}+1+(M-1) \log \left(\frac{n_0e}{(\lfloor n_0^{1/2}\log n_0\rfloor-1)M}\right)-\frac{n_0}{\lfloor n_0^{1/2}\log n_0\rfloor}\right).&
\end{align*}
\end{lemma}

Get back to the proof. Under $\operatorname{P}_{\theta_N}$, since $C_1N\leq N_0\leq C_2N$ almost surely holds as $N\to\infty$, there is
\begin{align*}
    R_N(\theta_N)\overset{\operatorname{P}_{\theta_N}}{\to}& \operatorname{plim} \frac{\ind(C_1N\leq N_0\leq C_2N)}{\sqrt{N}} \sum_{i=1}^{N}\left(2 W_i-1\right)&\\
&\times\left(\bar{\mu}_{\theta_N}\left(1-W_i, p(X_i;\theta_N)\right)-\frac{1}{M} \sum_{j \in \mathcal{J}_{M,\theta_N}(i)} \bar{\mu}_{\theta_N}\left(1-W_i, p(X_j;\theta_N)\right)\right),&
\end{align*}
where ``$\operatorname{plim}$'' represents the probability limit under $\operatorname{P}_{\theta_N}$. Then, using Lemma \ref{lem:2} with $p(X_i;\theta_N)$ now represented as $\tilde{X}$'s, we know that, as long as $N\geq 4/C_1$,
\begin{align*}
&\E_{\theta_N}\left[\left|\frac{\ind(C_1N\leq N_0\leq C_2N)}{\sqrt{N}} \sum_{i:W_i=1}\left(2 W_i-1\right)\right.\right.&\\
&\left.\left.~~\times\left(\bar{\mu}_{\theta_N}\left(1-W_i, p(X_i;\theta_N)\right)-\frac{1}{M} \sum_{j \in \mathcal{J}_{M,\theta_N}(i)} \bar{\mu}_{\theta_N}\left(1-W_i, p(X_j;\theta_N)\right)\right)\right|\mid \boldsymbol{W}_N\right]&\\
\leq&C_{\bar{\mu}}\E_{\theta_N}\left[\frac{\ind(C_1N\leq N_0\leq C_2N)}{\sqrt{N}} \sum_{i:W_i=1}\frac{1}{M} \sum_{j \in \mathcal{J}_{M,\theta_N}(i)} \left|p(X_j;\theta_N)-p(X_i;\theta_N)\right|\mid \boldsymbol{W}_N\right]&\\
\leq &\ind(C_1N\leq N_0\leq C_2N) C_{\bar{\mu}}\left[\frac{3\max_p\{\bar{r}_{\theta_N}(p)\}N_{1}}{N^{1 / 2}\lfloor N_0^{1/2}\log N_0\rfloor}\right.&\\
&\left.+\exp\left(\log \frac{MN_1}{N^{1/2}}+1+(M-1) \log \left(\frac{N_0e}{(\lfloor N_0^{1/2}\log N_0\rfloor-1)M}\right)-\frac{N_0}{\lfloor N_0^{1/2}\log N_0\rfloor}\right)\right],&
\end{align*}
where $\boldsymbol{W}_N:=(W_{1},\cdots,W_{N})$, $C_{\bar{\mu}}$ is the Lipschitz constant introduced in Assumption \ref{assumption:3}(iv), and $\bar{r}_{\theta_N}$ is the ratio of the density of $p(X;\theta_N)$ conditional on $W=1$ and $W=0$. Since
\begin{align}
    \bar{r}_{\theta_N}(p)&=\frac{\operatorname{P}_{\theta_N}(W=1|p(X;\theta_N)=p)/\operatorname{P}_{\theta_N}(W=1)}{\operatorname{P}_{\theta_N}(W=0|p(X;\theta_N)=p)/\operatorname{P}_{\theta_N}(W=0)}\notag\\
    &=\frac{p(1-q_{\theta_N})}{(1-p)q_{\theta_N}}\label{eq:ratio}
\end{align}
and the supports of $p(X;\theta_N)$ for different $N$ are uniformly bounded away from 0 and 1, $\bar{r}_{\theta_N}(p)$ is uniformly bounded in $N$ and $p$. This, in combination with the fact that $M=O(N^v)$ for some $v<1/2$, tells us that 
\begin{align*}
    &\E_{\theta_N}\left[\left|\frac{\ind(C_1N\leq N_0\leq C_2N)}{\sqrt{N}} \sum_{i:W_i=1}\left(2 W_i-1\right)\right.\right.&\\
&\left.\left.\times\left(\bar{\mu}_{\theta_N}\left(1-W_i, p(X_i;\theta_N)\right)-\frac{1}{M} \sum_{j \in \mathcal{J}_{M,\theta_N}(i)} \bar{\mu}_{\theta_N}\left(1-W_i, p(X_j;\theta_N)\right)\right)\right|\mid \boldsymbol{W}_N\right]\to 0
\end{align*}
uniformly in $\boldsymbol{W}_N$. By symmetry, we also have
\begin{align*}
    &\E_{\theta_N}\left[\left|\frac{\ind(C_1N\leq N_0\leq C_2N)}{\sqrt{N}} \sum_{i:W_i=0}\left(2 W_i-1\right)\right.\right.&\\
&\left.\left.\times\left(\bar{\mu}_{\theta_N}\left(1-W_i, p(X_i;\theta_N)\right)-\frac{1}{M} \sum_{j \in \mathcal{J}_{M,\theta_N}(i)} \bar{\mu}_{\theta_N}\left(1-W_i, p(X_j;\theta_N)\right)\right)\right|\mid \boldsymbol{W}_N\right]\to 0
\end{align*}
uniformly in $\boldsymbol{W}_N$. Accordingly
\begin{align*}
    &\E_{\theta_N}\left[\left|\frac{\ind(C_1N\leq N_0\leq C_2N)}{\sqrt{N}} \sum_{i=1}^{N}\left(2 W_i-1\right)\right.\right.&\\
&\left.\left.\times\left(\bar{\mu}_{\theta_N}\left(1-W_i, p(X_i;\theta_N)\right)-\frac{1}{M} \sum_{j \in \mathcal{J}_{M,\theta_N}(i)} \bar{\mu}_{\theta_N}\left(1-W_i, p(X_j;\theta_N)\right)\right)\right|\right]\to 0.
\end{align*}
As a corollary,
\begin{align*}
    &\frac{\ind(C_1N\leq N_0\leq C_2N)}{\sqrt{N}} \sum_{i=1}^{N}\left(2 W_i-1\right)&\\
&\times\left(\bar{\mu}_{\theta_N}\left(1-W_i, p(X_i;\theta_N)\right)-\frac{1}{M} \sum_{j \in \mathcal{J}_{M,\theta_N}(i)} \bar{\mu}_{\theta_N}\left(1-W_i, p(X_j;\theta_N)\right)\right)=o_{\operatorname{P}_{\theta_N}}(1).
\end{align*}
This completes the proof.
\end{proof}
\begin{remark}
    Actually, this lemma still holds when we replace $\theta_N$ with any sequence $\tilde{\theta}_N\to\theta^*$. Furthermore, the convergence of $R_N(\tilde{\theta}_N)$ under $\operatorname{P}_{\tilde{\theta}_N}$ is uniform in a family of $\tilde{\theta}_N$ as long as $\tilde{\theta}_N\to\theta^*$ uniformly. The use of the constant $C_{\bar{\mu}}$ here makes it necessary to enforce Assumption \ref{assumption:3}(iv)(a) in the article.
\end{remark}

\begin{proof}[Proof of Lemma \ref{lem:7}]
According to \eqref{eq:ratio} in the proof of Lemma \ref{lem:3}, for any neighborhood of $\theta^*$, there is a constant $C_h>0$ s.t. $1/C_h\leq \bar{r}_{\theta}(p)\leq C_h$ holds for any $\theta$ there. Then, according to Lemma S.8 in \cite{abadie2016matching}, we know that as long as $C_1N\leq N_0\leq C_2N$, $M\leq N_0$ and $M\leq N_1$, 
\begin{align*} 
\E_{\theta_N}\left[\left(\frac{K_{M,\theta_N}(i)}{M}\right)^{k} \mid W_i=1, N_0\right]\leq& \sum_{r=0}^{k} S(k, r) \left(\frac{N_0}{N_1}\right)^r C_{h}^{r} \frac{(2 M+r-1) !}{(2 M-1)!M^k}\\
\leq&\sum_{r=0}^{k} S(k, r) \left(\frac{N_0}{N_1}\right)^r C_{h}^{r} (r+1)^r,
\end{align*}
where $S(k,r)$ is the Stirling number of the second kind. By symmetry, we can bound the conditional expectation with regard to $W_i=0$, which yields this lemma.
\end{proof}
\begin{remark}
    Actually, the boundedness here is uniform with regard to the $\theta$ in a neighborhood of $\theta^*$. 
\end{remark}

\begin{proof}[Proof of Lemma \ref{lem:9}]
To prove this lemma, we need the following result.

\begin{lemma}[A modified version of Lemma S.6 in \cite{abadie2016matching}]\label{lem:5}
For each $N$, let  $\xi_{1: N}, \ldots, \xi_{N: N}$ be the order statistics for a random sample of size  $N$ from the uniform distribution over $[0,1]$ and let $F_N:[a_N, b_N] \to[0,1]$  be a strictly increasing and absolutely continuous distribution function with derivative $f_N(x)$. Suppose $f_N(\cdot)$'s are uniformly bounded and bounded away from 0 in $x$ and $N$. Suppose  $m_N:[a_N, b_N] \to \mathbb{R}$ are equicontinuous and uniformly bounded. Then, for any $M$ satisfying $M/N\to 0$ and $M\to\infty$ as $N\to\infty$, we have
\begin{align}
&\frac{1}{2M}\sum_{i=M+1}^{N-M} m_N\left(F_N^{-1}\left(\xi_{i: N}\right)\right)\left(\xi_{i+M: N}-\xi_{i-M: N}\right) \stackrel{\P}{\rightarrow} \int_{a_N}^{b_N} m_N(s) f_N(s) {\sf d} s\label{t1}&
\end{align}
and
\begin{align}
&\frac{1}{4M^2}\sum_{i=M+1}^{N-M} m_N\left(F_N^{-1}\left(\xi_{i: N}\right)\right) N\left(\xi_{i+M: N}-\xi_{i-M: N}\right)^{2}\stackrel{\P}{\rightarrow} \int_{a_N}^{b_N} m_N(s) f_N(s) {\sf d} s\label{t2}.&
\end{align}
\end{lemma}
\begin{lemma}[An enhanced version of Lemma S.7 in \cite{abadie2016matching}]\label{lem:6}
Suppose $M=O(N^v)$ for some $v<1/2$. For each $N$, let  $\operatorname{P}_N$  be a distribution with interval support $[a_N, b_N]$, distribution function  $G_N$, and density function  $g_N$  that is continuous on  $[a_N, b_N]$. Let $\tilde{X}_{1}, \ldots, \tilde{X}_{N}$ be independently sampled from $\operatorname{P}_N$, and let  $\tilde{X}_{j: N}$  be the corresponding $j$-th order statistic. Let $\tilde{X}_{j: N}=a_N$  if  $j<1$  and  $\tilde{X}_{j: N}=b_N$  if  $j>N$. Let  $V_{N k}$  be the rank of the $k$-th observation in $\tilde{X}_{1}, \ldots, \tilde{X}_{N}$. Let  $\tilde{P}_{N k}$  be  defined as the probability that  the $k$-th observation  will be a match for an out-of sample observation with continuous density $f_N$:
\begin{align*}
\tilde{P}_{N k} := & \int_{a_N}^{\left(\tilde{X}_{V_{N k}: N}+\tilde{X}_{V_{N k}+M: N}\right) / 2} f_N(x) {\sf d} x \ind(V_{N k} \leq M)&\\
& +\int_{\left(\tilde{X}_{V_{N k}: N}+\tilde{X}_{V_{N k}-M: N}\right) / 2}^{\left(\tilde{X}_{V_{N k}: N}+\tilde{X}_{V_{N k}+M: N}\right) / 2} f_N(x) {\sf d} x \ind(M+1 \leq V_{N k} \leq N-M)&\\
& +\int_{\left(\tilde{X}_{V_{N k}: N}+\tilde{X}_{V_{N k}-M: N}\right) / 2}^{b_N} f_N(x) {\sf d} x \ind(N-M+1 \leq V_{N k}).&
\end{align*}
Assume that $f_N(x)$ and $g_N(x)$ are uniformly bounded and bounded away from 0 in $x$ and $N$ on $[a_N, b_N]$. Let  $m_N(\cdot)$  be uniformly bounded functions with domain $[a_N, b_N]$. Then, under $\operatorname{P}_{N}$,
\begin{align*}
&\sum_{k=1}^{N} m_N\left(\tilde{X}_k\right)\left(\frac{\tilde{P}_{N k}}{M}-\frac{f_N\left(\tilde{X}_k\right)}{g_N\left(\tilde{X}_k\right)} \frac{G_N\left(\tilde{X}_{V_{N k}+M: N}\right)-G_N\left(\tilde{X}_{V_{N k}-M: N}\right)}{2M}\right) =o_{\operatorname{P}_N}(1)& 
\end{align*}
and
\begin{align*}
&\sum_{k=1}^{N} m_N\left(\tilde{X}_k\right) N\left(\left(\frac{\tilde{P}_{N k}}{M}\right)^{2}-\left(\frac{f_N\left(\tilde{X}_k\right)}{g_N\left(\tilde{X}_k\right)} \frac{G_N\left(\tilde{X}_{V_{N k}+M: N}\right)-G_N\left(\tilde{X}_{V_{N k}-M: N}\right)}{2M}\right)^{2}\right) =o_{\operatorname{P}_N}(1).&
\end{align*}
\end{lemma}
\begin{lemma}[A modified version of Lemma S.10 in \cite{abadie2016matching}] \label{lem:8}
Assume Assumptions \ref{assumption:1}-\ref{assumption:4} hold. Assume that $\bar{m}(w,p;\theta_N)$ defined in Lemma \ref{lem:9} is uniformly bounded and equicontinuous. For each $N$, let $P_{N k}$ be the probability that, among the $N$ samples, the $k$-th observation is used as a match for any particular observation in the opposite treatment arm, conditional on $ \mathbf{F}_{N,W_k}=(p(X_i;\theta_N))_{i:W_i=W_k}$, i.e.,
\begin{align*}
P_{N k}:=\operatorname{P}_{\theta_N}\left(k \in \mathcal{J}_{M,\theta_N}(i)|W_i=1-W_k,\mathbf{F}_{N,W_k}\right).
\end{align*}
Assume $M\to\infty$ as $N\to\infty$ and $M/N\to0$. Then, we have
\begin{align*}
&\frac{\ind(C_1N\leq N_0 \leq C_2N)}{N_{w}} \sum_{i: W_i=w} \bar{m}\left(w, p(X_i;\theta_N)\right) \left(\frac{K_{M,\theta_N}(i)}{M}-N_{1-w} \frac{P_{N i}}{M}\right)=o_{\operatorname{P}_{\theta_N}}(1)&
\end{align*}
and
\begin{align*}
&\frac{\ind(C_1N\leq N_0 \leq C_2N)}{N_{w}} \sum_{i: W_i=w} \bar{m}\left(w, p(X_i;\theta_N)\right) \left(\left(\frac{K_{M,\theta_N}(i)}{M}\right)^{2}-N_{1-w}^{2} \left(\frac{P_{N i}}{M}\right)^{2}\right)=o_{\operatorname{P}_{\theta_N}}(1).
\end{align*}
\end{lemma}

Get back to the proof of Lemma \ref{lem:9}. We will combine Lemmas \ref{lem:5}, \ref{lem:6}, and \ref{lem:8} to obtain this lemma, with $p(X_i;\theta_N)$ now viewed as the covariate in these lemmas. Let $g_w(p)$ be the probability density function of $p(X;\theta_N)$ conditional on $W=w$. Let $g(p)$ be the probability density function of $p(X;\theta_N)$. Since 
\begin{align*}
    g_0(p)=\frac{g(p)(1-p)}{1-q_{\theta_N}} ~~~{\rm and}~~~g_1(p)=\frac{g(p)p}{q_{\theta_N}},
\end{align*}
$g_w(p)$ is uniformly bounded and bounded away from 0 in $p$, $w$ and $N$. Therefore, the conditions of Lemma \ref{lem:5}, Lemma \ref{lem:6}, and Lemma \ref{lem:8} are satisfied and we then obtain
\begin{align*}
    &\frac{\ind(C_1N\leq N_0 \leq C_2N)}{N_{w}} \sum_{i: W_i=w} \bar{m}(w,p(X_i;\theta_N)) \frac{K_{M,\theta_N}(i)}{M}\notag \\
    \stackrel{\operatorname{P}_{\theta_N}}{\rightarrow}&\frac{\ind(C_1N\leq N_0\leq C_2N)N_{1-w}}{N_w}\E_{\theta_N}\left[\bar{m}(w,p(X;\theta_N))\left(\frac{g_\ind(p(X;\theta_N))}{g_0(p(X;\theta_N))}\right)^{1-2 w} \mid W=w\right]&\\
    &\times \frac{\ind(C_1N\leq N_0 \leq C_2N)}{N_{w}} \sum_{i: W_i=w} \bar{m}(w,p(X_i;\theta_N)) \left(\frac{K_{M,\theta_N}(i)}{M}\right)^{2}\notag\\
    \stackrel{\operatorname{P}_{\theta_N}}{\rightarrow}&\left(\frac{\ind(C_1N\leq N_0\leq C_2N)N_{1-w}}{N_w}\right)^2\E_{\theta_N}\left[\bar{m}(w,p(X;\theta_N))\left(\frac{g_\ind(p(X;\theta_N))}{g_0(p(X;\theta_N))}\right)^{2(1-2 w)} \mid W=w\right]&
\end{align*}
under $\operatorname{P}_{\theta_N}$. The above results, when combined with the fact that
\begin{align*}
    \frac{g_{1}(p)}{g_0(p)}=\frac{p(1-q_{\theta_N})}{(1-p)q_{\theta_N}}~~~{\rm and}~~~
    \ind(C_1N\leq N_0\leq C_2N)\frac{N_{1-w}}{N_w}\overset{\operatorname{P}_{\theta_N}}{\to}\left(\frac{q_{\theta_N}}{1-q_{\theta_N}}\right)^{1-2w},
\end{align*}
leads to the desired result.
\end{proof}
\begin{remark}
    It's based on Lemma S.11 in \cite{abadie2016matching}. Actually, this lemma still holds when we replace $\theta_N$ with any sequence $\tilde{\theta}_N\to\theta^*$. Furthermore, the convergences obtained in this lemma are uniform in a family of $\tilde{\theta}_N$ as long as $\tilde{\theta}_N\to\theta^*$ uniformly.
\end{remark}

\begin{proof}[Proof of Lemma \ref{lem:10}]

To prove this lemma, we use the martingale representation of matching estimators as in \cite{abadie2016matching}. Employing the Cram\'{e}r-Wold device and consider $C_{N}=z_{1} D_{N}\left(\theta_{N}\right)+z_{2}^{\prime} \Delta_{N}\left(\theta_{N}\right)$ with the understanding by \cite{abadie2016matching} that
\begin{align*}
C_{N}= & z_{1} \frac{1}{\sqrt{N}} \sum_{i=1}^{N}\left(\bar{\mu}_{\theta_{N}}\left(1, p(X_i;\theta_N)\right)-\bar{\mu}_{\theta_{N}}\left(0, p(X_i;\theta_N)\right)-\tau\right) &\\
& +z_{1} \frac{\ind(C_1N\leq N_0\leq C_2N)}{\sqrt{N}} \sum_{i=1}^{N}\left(2 W_{i}-1\right)\left(1+\frac{K_{M, \theta_{N}}(i)}{M}\right) &\\
&~~~~ \times\left(Y_{i}-\bar{\mu}_{\theta_{N}}\left(W_{i}, p(X_i;\theta_N)\right)\right) &\\
& +z_{2}^{\prime} \frac{1}{\sqrt{N}} \sum_{i=1}^{N} X_{i} \frac{W_{i}-p(X_i;\theta_N)}{p(X_i;\theta_N)\left(1-p(X_i;\theta_N)\right)} f\left(X_{i}^{\prime} \theta_{N}\right). &
\end{align*}

Let $C_{N}=\sum_{k=1}^{3 N} \xi_{k}$ with $\xi_k$'s taking the values in \citet[Pages 802-803]{abadie2016matching}. By using Holder's Inequality, Lemma \ref{lem:7}, and our Assumption \ref{assumption:3}(iv), it holds true that
\begin{align*}
&\sum_{k=1}^{3 N} \E_{\theta_{N}}\left[\left|\xi_{k}\right|^{2+\delta}\right]\to 0&
\end{align*}
for some $\delta>0$. As a result, we obtain (by leveraging the martingale central limit theorem) that, under  $\operatorname{P}_{\theta_{N}}$,
\begin{align*}
&C_{N} \stackrel{d}{\rightarrow} N\left(0, \sigma_{1}^{2}+\sigma_{2}^{2}+\sigma_{3}^{2}\right),&
\end{align*}
where
\begin{align*}
\sigma_{1}^{2}=&\operatorname{plim} \sum_{k=1}^{N} \E_{\theta_{N}}\left[\xi_{k}^{2} \mid \mathcal{F}_{k-1}\right],& \\
\sigma_{2}^{2}=&\operatorname{plim} \sum_{k=N+1}^{2 N} \E_{\theta_{N}}\left[\xi_{k}^{2} \mid \mathcal{F}_{k-1}\right],& \\
\sigma_{3}^{2}=&\operatorname{plim} \sum_{k=2 N+1}^{3 N} \E_{\theta_{N}}\left[\xi_{k}^{2} \mid \mathcal{F}_{k-1}\right],&
\end{align*}
where $\operatorname{plim}$ means the probability limit under $\operatorname{P}_{\theta_N}$. Same as \cite{abadie2016matching}, we have
\begin{align*}
\sigma_{1}^{2}= & z_{1}^{2} \E\left[\left(\bar{\mu}\left(1, p(X)\right)-\bar{\mu}\left(0, p(X)\right)-\tau\right)^{2}\right]  +z_{2}^{\prime} \E\left[\frac{f^{2}\left(X^{\prime} \theta^{*}\right)}{p(X)\left(1-p(X)\right)} \E\left[X \mid p(X)\right] \E\left[X^{\prime} \mid p(X)\right]\right] z_{2}.&
\end{align*}
By Lemma \ref{lem:9} and the fact that $W$  is independent of $X$ given  $p(X;\theta_N)$ under $\operatorname{P}_{\theta_N}$, we have
\begin{align*}
\sigma_{2}^{2}= & z_{2}^{\prime} \E\left[\frac{f^{2}\left(X^{\prime} \theta^{*}\right)}{p(X)\left(1-p(X)\right)} \operatorname{Var}\left(X \mid p(X)\right)\right] z_{2}& \\
& +z_{1}^{2} \operatorname{plim}\frac{N_1}{N}\E_{\theta_N}\left[\left(\mu(1, X)-\bar{\mu}_{\theta_N}(1,p(X; \theta_{N})\right)^2\left(1+\frac{1-p(X;\theta_N)}{p(X;\theta_N)}\right)^2\mid W=1\right]&\\
&+z_{1}^{2} \operatorname{plim}\frac{N_0}{N}\E_{\theta_N}\left[\left(\mu(0, X)-\bar{\mu}_{\theta_N}(0,p(X; \theta_{N}))\right)^2\left(1+\frac{p(X;\theta_N)}{1-p(X;\theta_N)}\right)^2\mid W=0\right]&\\
&+2z_2'z_1 \operatorname{plim}\frac{N_1}{N}\E_{\theta_N}\left[ \left(X-\E_{\theta_N}\left[X \mid p(X;\theta_N)\right]\right)\left(\mu(1, X)-\bar{\mu}_{\theta_N}(1,p(X; \theta_{N}))\right)\right. &\\
& \left.\times \frac{p(X;\theta_{N}) f\left(X^{\prime} \theta_{N}\right)}{p(X;\theta_N)\left(1-p(X;\theta_N)\right)}\left(1+\frac{1-p(X;\theta_N)}{p(X;\theta_N)}\right)\mid W=1\right]& \\
&+2z_2'z_1 \operatorname{plim}\frac{N_0}{N}\E_{\theta_N}\left[ \left(X-\E_{\theta_N}\left[X \mid p(X;\theta_N)\right]\right)\left(\mu(0, X)-\bar{\mu}_{\theta_N}(0,p(X; \theta_{N}))\right)\right. &\\
& \left.\times \frac{(1-p(X;\theta_N)) f\left(X^{\prime} \theta_{N}\right)}{p(X;\theta_N)\left(1-p(X;\theta_N)\right)}\left(1+\frac{p(X;\theta_N)}{1-p(X;\theta_N)}\right)\mid W=0\right]& \\
= & z_{2}^{\prime} \E\left[\frac{f^{2}\left(X^{\prime} \theta^{*}\right)}{p(X)\left(1-p(X)\right)} \operatorname{Var}\left(X \mid p(X)\right)\right] z_{2}& \\
& +z_{1}^{2} \operatorname{plim} \E_{\theta_N}\left[\left(\mu(1, X)-\bar{\mu}_{\theta_N}(1,p(X; \theta_{N}))\right)^2\left(\frac{1}{p(X;\theta_N)}\right)^2W\right]&\\
&+z_{1}^{2} \operatorname{plim} \E_{\theta_N}\left[\left(\mu(0, X)-\bar{\mu}_{\theta_N}(0,p(X; \theta_{N}))\right)^2\left(\frac{1}{1-p(X;\theta_N)}\right)^2(1-W)\right]&\\
&+2z_2'z_1 \operatorname{plim} \E_{\theta_N}\left[\left(X-\E_{\theta_N}\left[X \mid p(X;\theta_N)\right]\right)\left(\mu(1, X)-\bar{\mu}_{\theta_N}(1,p(X; \theta_{N}))\right)\right. &\\
& \left.\times \frac{p(X;\theta_{N}) f\left(X^{\prime} \theta_{N}\right)}{p(X;\theta_N)\left(1-p(X;\theta_N)\right)}\frac{W}{p(X;\theta_N)}\right]& \\
&+2z_2'z_1 \operatorname{plim} \E_{\theta_N}\left[ \left(X-\E_{\theta_N}\left[X \mid p(X;\theta_N)\right]\right)\left(\mu(0, X)-\bar{\mu}_{\theta_N}(0,p(X; \theta_{N}))\right)\right. &\\
& \left.\times \frac{(1-p(X;\theta_N)) f\left(X^{\prime} \theta_{N}\right)}{p(X;\theta_N)\left(1-p(X;\theta_N)\right)}\frac{1-W}{1-p(X;\theta_N)}\right]& \\
=&z_{2}^{\prime} \E\left[\frac{f^{2}\left(X^{\prime} \theta^{*}\right)}{p(X)\left(1-p(X)\right)} \operatorname{Var}\left(X \mid p(X)\right)\right] z_{2}&\\
&+z_1^2\E\left[\frac{\operatorname{Var}\left(\mu(1, X) \mid p(X)\right)}{1-p(X)}+\frac{\operatorname{Var}\left(\mu(0, X) \mid p(X)\right)}{1-p(X)}\right] \\
& +2 z_{2}^{\prime} \E\left[\left(\frac{\operatorname{Cov}\left(X, \mu(1, X) \mid p(X)\right)}{p(X)}\right.\right. \\
& \left.\left.+\frac{\operatorname{Cov}\left(X, \mu(0, X) \mid p(X)\right)}{1-p(X)}\right) p(X)\right] z_{1} .
\end{align*}
As for $\sigma_3^2$, by the analysis in \cite{abadie2016matching} and the methods used in the analysis about $\sigma_2^2$, we know that
\begin{align*}
\sigma_{3}^{2}= & z_{1}^{2} \operatorname{plim} \frac{1}{N} \sum_{i=1}^{N}\left(1+\frac{K_{M, \theta_{N}}(i)}{M}\right)^{2}\sigma^2(W_i, X_i)& \\
= & z_{1}^{2} \operatorname{plim} \frac{1}{N} \sum_{i=1}^{N}\left(1+\frac{K_{M, \theta_{N}}(i)}{M}\right)^{2}\E_{\theta_N}[\sigma^2(W_i, X_i)|p(X_i;\theta_N)]& \\
= & z_{1}^{2} \E\left[\frac{\E\left[\operatorname{Var}(Y \mid X, W=1) \mid p(X)\right]}{p(X)}+\frac{\E\left[\operatorname{Var}(Y \mid X, W=0) \mid p(X)\right]}{1-p(X)}\right].
\end{align*}
Putting them together, we obtain
\begin{align*}
&\sigma_{1}^{2}+\sigma_{2}^{2}+\sigma_{3}^{2}=z_{1}^{2} \sigma^{2}+z_{2}^{\prime} I_{\theta^{*}} z_{2}+2 z_{2}^{\prime} c z_{1}.&   
\end{align*}
Combining all, under $\operatorname{P}_{\theta_{N}}$, we obtain
\begin{align*}
&\left(\begin{array}{c}
D_{N}\left(\theta_{N}\right) \\
\Delta_{N}\left(\theta_{N}\right)
\end{array}\right) \stackrel{d}{\rightarrow} N\left(\left(\begin{array}{l}
0 \\
0
\end{array}\right),\left(\begin{array}{cc}
\sigma^{2} & c^{\prime} \\
c & I_{\theta^{*}}
\end{array}\right)\right).&
\end{align*}
This completes the proof.
\end{proof}
\begin{remark}
    Actually, this lemma still holds when we replace $\theta_N$ with any sequence $\tilde{\theta}_N\to\theta^*$. Furthermore, the convergence of $(D_N(\tilde{\theta}_N),\delta_N(\tilde{\theta}_N))$ under $\operatorname{P}_{\tilde{\theta}_N}$ is uniform in a family of $\tilde{\theta}_N$ as long as $\tilde{\theta}_N\to\theta^*$ uniformly.
\end{remark}
\section{Proof of the lemmas in Section \ref{sec:proof_direct}}\label{sec:proof_indirect}
\begin{proof}[Proof of Lemma \ref{lem:2}]
Divide the support of  $F_{0}$  and  $F_{1}$  into  $K$  intervals of equal probability,  $1/K$, under  $F_{0}$. With a little bit abuse of notation, let  $Z_{M, n_{0}}$  be the number of intervals that are not occupied by at least  $M$  observations from the sample $\tilde{X}_{0,1}, \cdots, \tilde{X}_{0, n_{0}}$. Let  $\mu_{M, n_{0}}=\E\left[Z_{M, n_{0}}\right]$. Then, as long as $n_0\geq K-1$,
\begin{align*}
\mu_{M, n_{0}} & =\sum_{m=0}^{M-1}K\left(\begin{array}{c}
n_{0} \\
m
\end{array}\right)\left(\frac{1}{K}\right)^{m}\left(1-\frac{1}{K}\right)^{n_{0}-m} &\\
& \leq \sum_{m=0}^{M-1}K\frac{n_{0}^{m}}{m !}\left(\frac{1}{K}\right)^{m}\left(1-\frac{1}{K}\right)^{n_{0}-m}&\\
& \leq M \frac{n_{0}^{M-1}e^M}{M^MK^{M-1}}\left(1-\frac{1}{K}\right)^{n_{0}-M+1}&\\
&\leq M \frac{n_{0}^{M-1}e^M}{(K-1)^{M-1}M^M}\exp\left(-\frac{n_0}{K}\right)&\\
&=\exp\left(\log M+1+(M-1) \log \left(\frac{n_0e}{(K-1)M}\right)-\frac{n_0}{K}\right).&
\end{align*}
According to Markov's inequality, there is
\begin{align}\label{lem2.1}
&\operatorname{P}\left(Z_{M, n_{0}}>0\right) \leq \mu_{M, n_{0}} \leq \exp\left(\log M+1+(M-1) \log \left(\frac{n_0e}{(K-1)M}\right)-\frac{n_0}{K}\right).&
\end{align}
For  $0 \leq k \leq K$ , let  $c_{n_{0}, k}$  be the point in the support of  $F_{0}$  such that  $c_{n_{0}, k}=F_{0}^{-1}\left(k /K\right)$. Especially, let $c_{n_{0}, -1}=0$ and $c_{n_0,n_0+1}=1$. Then, if the support of  $F_{0}$  and  $F_{1}$  is an interval inside $[0,1]$, we have
\begin{align*}
&\E\left[\sum_{m=1}^{M}\left|U_{n_{0}, n_{1}, i}\right|_{(m)} \mid Z_{M, n_{0}}=0\right]&\\
\leq &\sum_{k=1}^{K} M\left(c_{n_{0}, k+1}-c_{n_{0}, k-2}\right) \operatorname{P}\left(c_{n_{0}, k-1} \leq \tilde{X}_{1, i} \leq c_{n_{0}, k} \mid Z_{M, n_{0}}=0\right)&\\
\leq &\frac{M \bar{r}}{K} \sum_{k=1}^{K}\left(c_{n_{0}, k+1}-c_{n_{0}, k-2}\right)&\\
\leq &\frac{3M \bar{r}}{K},&
\end{align*}
and furthermore,
\begin{align}
&\E\left[\frac{1}{\sqrt{N}} \sum_{i=1}^{n_{1}} \frac{1}{M} \sum_{m=1}^{M}\left|U_{n_{0}, n_{1}, i}\right|_{(m)}\right]\notag&\\
=&\E\left[\frac{1}{\sqrt{N}} \sum_{i=1}^{n_{1}} \frac{1}{M} \sum_{m=1}^{M}\left|U_{n_{0}, n_{1}, i}\right|_{(m)} \mid Z_{M, n_{0}}=0\right] \operatorname{P}\left(Z_{M, n_{0}}=0\right)& \notag\\
&+\E\left[\frac{1}{\sqrt{N}} \sum_{i=1}^{n_{1}} \frac{1}{M} \sum_{m=1}^{M}\left|U_{n_{0}, n_{1}, i}\right|_{(m)} \mid Z_{M, n_{0}}>0\right] \operatorname{P}\left(Z_{M, n_{0}}>0\right)& \notag\\
\leq &\E\left[\frac{1}{\sqrt{N}} \sum_{i=1}^{n_{1}} \frac{1}{M} \sum_{m=1}^{M}\left|U_{n_{0}, n_{1}, i}\right|_{(m)} \mid Z_{M, n_{0}}=0\right]+\frac{n_{1}}{N^{1 / 2}} \operatorname{P}\left(Z_{M, n_{0}}>0\right)& \notag\\
=&\frac{1}{\sqrt{N}} \sum_{i=1}^{n_{1}} \E\left[\frac{1}{M} \sum_{m=1}^{M}\left|U_{n_{0}, n_{1}, i}\right|_{(m)} \mid Z_{M, n_{0}}=0\right]+\frac{n_{1}}{N^{1 / 2}} \operatorname{P}\left(Z_{M, n_{0}}>0\right) \notag&\\
\leq &\frac{n_{1}}{N^{1 / 2}}\left[\frac{3\bar{r}}{K}+\exp\left(\log M+1+(M-1) \log \left(\frac{n_0e}{(K-1)M}\right)-\frac{n_0}{K}\right)\right]\label{lem2.2}.&
\end{align}
Letting $K=\lfloor n_0^{1/2}\log n_0\rfloor\leq n_0$ in (\ref{lem2.1}) and (\ref{lem2.2}),  the lemma is proved.
\end{proof}
\begin{remark}
  Equation  \eqref{lem2.1} can also lead to \eqref{eq:use_l2} in the proof of Theorem \ref{thm:var} and \eqref{eq:I2_use2} in the proof of Lemma \ref{lem:6}.
\end{remark}
\begin{proof}[Proof of Lemma \ref{lem:5}]
In this proof, we use the following lemma about some regular distributions without a proof.
\begin{lemma}\label{lem:4}
Let  $Y_{1}, \ldots, Y_{N+1}$  be independently sampled from a standard exponential (i.e.,  $\Gamma(1,1)$), the Gamma distribution with parameters  $(1,1)$. Let  $S_{j}=\sum_{i=1}^{j} Y_{i}$  and  $S_{N+1, j}=\sum_{i=1}^{j} Y_{i} / \sum_{i=1}^{N+1} Y_{i}$, for  $1 \leq j \leq   N+1$. Let $k$ denote a positive integer. Then, $S_{N+1, j}$  has a Beta distribution with parameters $(j, N-j+1)$ and moments
\begin{align*}
&\E\left[\left(S_{N+1, j}\right)^{k}\right]=\frac{(j+k-1) ! N !}{(j-1) !(N+k) !}&.
\end{align*}
Especially, Let  $\xi_{1: N}, \ldots, \xi_{N: N}$ be the order statistics for a random sample of size  $N$ from the uniform distribution over $[0,1]$, any $\xi_{k+j: N}-\xi_{k: N}$ can be viewed as $S_{N+1, j}$ and has a Beta distribution with parameters $(j, N-j+1)$.
\end{lemma}
Get back to the proof. From the assumptions, we can find that $F_N^{-1}$ are equicontinuous. To prove (\ref{t1}), we need the following three results:
\begin{align}
&\sum_{i=M+1}^{N-M}\left(m_N\left(F_N^{-1}\left(\xi_{i: N}\right)\right)-m_N\left(F_N^{-1}(i / N)\right)\right)\frac{\xi_{i+M: N}-\xi_{i-M: N}}{M}=o_{\P}(1),\label{1}&\\
&\sum_{i=M+1}^{N-M} m_N\left(F_N^{-1}(i / N)\right)\frac{\xi_{i+M: N}-\xi_{i-M: N}}{M}-\frac{2}{N} \sum_{i=M+1}^{N-M} m_N\left(F_N^{-1}(i / N)\right)=o_{\P}(1),\label{2}&\\
&\frac{1}{N} \sum_{i=M+1}^{N-M} m_N\left(F_N^{-1}(i / N)\right)-\int_{a}^{b} m_N(s) f_N(s) {\sf d} s=o(1).\label{3}&
\end{align}
For any $\epsilon>0$, there exists $\delta>0$ such that
\begin{align*}
&\operatorname{P}\left(\max _{i=1, \ldots, N}\left|m_N\left(F_N^{-1}\left(\xi_{i: N}\right)\right)-m_N\left(F_N^{-1}(i / N)\right)\right|<\varepsilon\right)\geq \operatorname{P}\left(\max _{i=1, \ldots, N}\left|\xi_{i: N}-i / N\right|<\delta\right) \rightarrow 1,&   
\end{align*}
and the first inequality holds for any $N$.  Then, we have
\begin{align*}
&\left|\sum_{i=M+1}^{N-M}\left(m_N\left(F_N^{-1}\left(\xi_{i: N}\right)\right)-m_N\left(F_N^{-1}(i / N)\right)\right)\frac{\xi_{i+M: N}-\xi_{i-M: N}}{M}\right|&\\
\leq&\max _{i=1, \ldots, N}\left|m_N\left(F_N^{-1}\left(\xi_{i: N}\right)\right)-m_N\left(F_N^{-1}(i / N)\right)\right|\left|\sum_{i=M+1}^{N-M}\frac{\xi_{i+M: N}-\xi_{i-M: N}}{M}\right|& \\
\leq&2\max _{i=1, \ldots, N}\left|m_N\left(F_N^{-1}\left(\xi_{i: N}\right)\right)-m_N\left(F_N^{-1}(i / N)\right)\right|\stackrel{\P}{\rightarrow} 0,   &
\end{align*}
that is, we have (\ref{1}). According to the derivation of Lemma S.6 in \cite{abadie2016matching}, we have
\begin{align*}
    &\E\left[\left(\sum_{i=M+1}^{N-M} m_N\left(F_N^{-1}(i / N)\right)\left(\xi_{i+k+1: N}-\xi_{i+k: N}\right)-\frac{1}{N} \sum_{i=M+1}^{N-M} m_N\left(F_N^{-1}(i / N)\right)\right)^2\right]&\\
    \leq &\left[\frac{-1}{N(N+1)} \sum_{i=M+1}^{N-M} m_N\left(F_N^{-1}(i / N)\right)\right]^2+\frac{1}{N^{2}} \sum_{i=M+1}^{N-M} m_N\left(F_N^{-1}(i / N)\right)^{2}&
\end{align*}
for $-M\leq k\leq M-1$. Then
\begin{align*}
    &\E\left[\left(\sum_{i=M+1}^{N-M} m_N\left(F_N^{-1}(i / N)\right)\frac{\xi_{i+M: N}-\xi_{i-M: N}}{M}-\frac{1}{N} \sum_{i=M+1}^{N-M} m_N\left(F_N^{-1}(i / N)\right)\right)^2\right]&\\
    \leq&\frac{1}{M^2}2M\sum_{k=-M}^{M-1}\E\left[\left(\sum_{i=M+1}^{N-M} m_N\left(F_N^{-1}(i / N)\right)\left(\xi_{i+k+1: N}-\xi_{i+k: N}\right)-\frac{1}{N} \sum_{i=M+1}^{N-M} m_N\left(F_N^{-1}(i / N)\right)\right)^2\right]&\\
    \leq &4\left[\frac{-1}{N(N+1)} \sum_{i=M+1}^{N-M} m_N\left(F_N^{-1}(i / N)\right)\right]^2+\frac{4}{N^{2}} \sum_{i=M+1}^{N-M} m_N\left(F_N^{-1}(i / N)\right)^{2}\\
    \leq& \frac{4C_m^2}{(N+1)^2} +\frac{4C_m^2}{N}\\
    \to& 0,
\end{align*}
where $C_m$ is the upper bound of $|m_N(\cdot)|$. Equation (\ref{2}) holds as its corollary. Finally, since $m_N$ and $F_N^{-1}$ are both equicontinuous, and $M/N\to0$, we have
\begin{align*}
    \frac{1}{N} \sum_{i=M+1}^{N-M} m_N\left(F_N^{-1}(i / N)\right)\overset{\P}{\to}\int_{0}^{1} m_N(F_N^{-1}(t)) {\sf d} t=\int_{a_N}^{b_N} m_N(s) f_N(s) {\sf d} s.  
\end{align*}
Putting the above results together, we obtain (\ref{t1}).

To prove (\ref{t2}), we need (\ref{3}) and the following two results,
\begin{align}
    &\sum_{i=M+1}^{N-M}\left(m\left(F_N^{-1}\left(\xi_{i: N}\right)\right)-m\left(F_N^{-1}(i / N)\right)\right) N\left(\frac{\xi_{i+M: N}-\xi_{i-M: N}}{M}\right)^{2}=o_{\P}(1),&\label{4}\\
    &\sum_{i=M+1}^{N-M} m\left(F_N^{-1}(i / N)\right) N\left(\frac{\xi_{i+M: N}-\xi_{i-M: N}}{M}\right)^{2}-\frac{4}{N} \sum_{i=M+1}^{N-M} m\left(F_N^{-1}(i / N)\right)=o_{\P}(1).&\label{5}
\end{align}
According to Lemma \ref{lem:4},
\begin{align*}
    &\E\left[\sum_{i=M+1}^{N-M} N\left(\frac{\xi_{i+M: N}-\xi_{i-M: N}}{M}\right)^{2}\right]=\frac{2 M(2 M+1)}{M^2}\frac{(N-2M)N}{(N+1)(N+2)}\leq 8.&
\end{align*}
Then, similar to the proof of (\ref{1}), we know that (\ref{4}) holds. According to the derivation of Lemma S.6 in \cite{abadie2016matching}, we have
\begin{align*}
    &E\left[\sum_{i=M+1}^{N-M} m\left(F_N^{-1}(i / N)\right) N\left(\frac{\xi_{i+M: N}-\xi_{i-M: N}}{M}\right)^{2}-\frac{4}{N} \sum_{i=M+1}^{N-M} m\left(F_N^{-1}(i / N)\right)\right]&\\
    \leq&\left[\left(\frac{2(2 M+1) N}{(N+1)(N+2)M}-\frac{4}{N}\right)\sum_{i=M+1}^{N-M} m\left(F_N^{-1}(i / N)\right)\right]^2&\\
    &+C_{m}^{2}\frac{(4M-1)(2M+3)!}{(2M-1)!M^4}\frac{N !(N-2M)N^{2}}{(N+4)!}&\\
    \leq&\left[\frac{2N^2-3NM-2M}{N(N+1)(N+2)M}\sum_{i=M+1}^{N-M} m\left(F_N^{-1}(i / N)\right)\right]^2+625C_{m}^{2}\frac{4M-1}{N+4}\\
    \leq&\left(\frac{(2N^2-3NM-2M)C_m}{(N+1)(N+2)M}\right)^2+625C_{m}^{2}\frac{4M-1}{N+4}\\
     \to&0.&
\end{align*}
Then, (\ref{5}) holds as the corollary of this result. Putting together, we proved (\ref{t2}). 
\end{proof}
\begin{remark}
Of note, when there are two families of $m_N(\cdot)$ and $f_N(\cdot)$ such that the functions there satisfy the conditions of this lemma uniformly, the convergences in this lemma hold uniformly for any $m_N$ and $f_N$ chosen from the two families.
\end{remark}
\begin{remark}
To meet the assumptions of this lemma in the calculation of $\sigma_2^2$ in Lemma \ref{lem:10}, we require the global Lipschitz constant in Assumption \ref{assumption:3}(iv) and $f$ to be continuous  in Assumption \ref{assumption:3}(ii). 
In this lemma, there is no distribution shift, and we thus simply use $o_{\P}(1)$ and $\overset{\P}{\to}$ to describe the convergence in probability.
\end{remark}

\begin{proof}[Prood of Lemma \ref{lem:6}] 
Let
\begin{align*}
    &Z_{N k}^{(1)}:=m_N\left(\tilde{X}_k\right) N\left(\frac{\tilde{P}_{N k}}{M}-\frac{f_N\left(\tilde{X}_k\right)}{g_N\left(\tilde{X}_k\right)} \frac{G_N\left(\tilde{X}_{V_{N k}+M: N}\right)-G_N\left(\tilde{X}_{V_{N k}-M: N}\right)}{2M}\right).&
\end{align*}
Since  $f_N$  and  $g_N$  are continuous, with the mean value theorem, there are values  $\bar{X}_{f, k, N, M}$  and  $\bar{X}_{g, k, N, M}$  in  
\[
\left(\left(\tilde{X}_{V_{N k}: N}+\tilde{X}_{V_{N k}-M: N}\right) / 2, \left(\tilde{X}_{V_{N k}: N}+\tilde{X}_{V_{N k}+M: N}\right) / 2\right)
\]
and  
\[
\left(\tilde{X}_{V_{N k}-M: N}, \tilde{X}_{V_{N k}+M: N}\right), 
\]
respectively, such that
\begin{align*}
0= &m_N\left(\tilde{X}_k\right) N\left(\frac{\tilde{P}_{N k}}{M}-\frac{f_N\left(\bar{X}_{f, k, N, M}\right)}{g_N\left(\bar{X}_{g, k, N, M}\right)} \frac{G_N\left(\tilde{X}_{V_{N k}+M: N}\right)-G_N\left(\tilde{X}_{V_{N k}-M: N}\right)}{2M}\right)+o_{\operatorname{P}_N}(1)&\\
=&Z_{N k}^{(1)}+\left(\frac{f_N\left(\tilde{X}_k\right)}{g_N\left(\tilde{X}_k\right)}-\frac{f_N\left(\bar{X}_{f, k, N, M}\right)}{g_N\left(\bar{X}_{g, k, N, M}\right)}\right)m_N \left(\tilde{X}_k\right)\frac{N\left(G_N\left(\tilde{X}_{V_{N k}+M: N}\right)-G_N\left(\tilde{X}_{V_{N k}-M: N}\right)\right)}{2M}+o_{\operatorname{P}_N}(1).&
\end{align*}
Since $g_N(x)$ is bounded and bounded away from 0,  Lemma \ref{lem:2}(i) implies that
\begin{align}
    &\max_{k}|\tilde{X}_{k:N}-\tilde{X}_{k+M:N}|=o_{\operatorname{P}_N}(1).\label{eq:I2_use2}
\end{align}  
According to lemma \ref{lem:4}, we know that
\begin{align}\label{double}
    &\E_N\left[\frac{N\left(G_N\left(\tilde{X}_{V_{N k}+M: N}\right)-G_N\left(\tilde{X}_{V_{N k}-M: N}\right)\right)}{2M}\right]=\frac{2MN}{2M(N+1)}=O(1),&
\end{align}  
where $\E_N$ is the expectation under $\operatorname{P}_N$. Therefore, 
\begin{align*}
0=&Z_{N k}^{(1)}+\left(\frac{f_N\left(\tilde{X}_{k}\right)}{g_N\left(\tilde{X}_k\right)}-\frac{f_N\left(\bar{X}_{f, k, N, M}\right)}{g_N\left(\bar{X}_{g, k, N, M}\right)}\right)m_N \left(\tilde{X}_{k}\right)\frac{N\left(G_N\left(\tilde{X}_{V_{N k}+M: N}\right)-G_N\left(\tilde{X}_{V_{N k}-M: N}\right)\right)}{2M}+o_{\operatorname{P}_N}(1)& \\
= & Z_{N k}^{(1)}+o_{\operatorname{P}_N}(1) O_{\operatorname{P}_N}(1)+o_{\operatorname{P}_N}(1),&
\end{align*}
that is, $Z_{N k}^{(1)}=o_{\operatorname{P}_N}(1)$. According to the derivation of Lemma S.7 in \cite{abadie2016matching}, for any $r>0$,
\begin{align}\label{P_NK}
\E_N\left[\left|\frac{N\tilde{P}_{N k}}{M}\right|^r\right]\leq C_h^r\frac{(2 M+r-1)!}{(2 M-1)!M^r}, &
\end{align}
where $C_h$ is an upper bound of $f_N(x)/g_N(x)$. With this result and (\ref{double}), we know that for any $r>0$, $\E_N\left[\left|Z_{Nk}^{(1)}\right|^r\right]$ are uniformly bounded in $M$ and $N$. This, when combined with the fact that $Z_{N k}^{(1)}=o_{\operatorname{P}_N}(1)$, implies 
\[
\E_N\left[\left|Z_{Nk}^{(1)}\right|\right]\to 0. 
\]
Then, according to the derivation of Lemma S.7 in \cite{abadie2016matching}, 
\begin{align*}
    &\frac{1}{N}\sum_{k=1}^{N} Z_{Nk}^{(1)}=o_{\operatorname{P}_N}(1).
\end{align*}
Since we also know that $\left(Z_{N k}^{(1)}\right)^2=o_{\operatorname{P}_N}(1)$ and that $E_N\left[\left|Z_{Nk}^{(1)}\right|^{2r}\right]$ is uniformly bounded in $M$ and $N$, there is $E_N\left[\left|Z_{Nk}^{(1)}\right|^2\right]\to 0$. This implies the second statement of this lemma.
\end{proof}
\begin{remark}
    Actually, if there is a family of $\operatorname{P}_N$, a family of $f_N(\cdot)$, and a family of $m_N(\cdot)$ such that all the $g_N$, $f_N$ and $m_N$ functions derived from the families satisfy the conditions of this lemma uniformly, the convergences in this lemma hold uniformly for any $\operatorname{P}_N$, $f_N$ and $m_N$ from the families. Lemma \ref{lem:6} is a stronger version of Lemma S.7 in \cite{abadie2016matching}. The original lemma is its corollary.
\end{remark}

\begin{proof}[Proof of Lemma \ref{lem:8}]
Here, we only prove the second formula. The proof of the first formula is similar but simpler. Without loss of generality, we only consider the cases where $w=0$. 

Let $\mathbf{W}_N=(W_1,\cdots,W_N)$. For any $i$ satisfying $W_i=0$, we have $K_{M,\theta_N}(i)\mid \mathbf{W}_N,\mathbf{F}_{N,0}\sim B(N_1,P_{N i})$, where $B(\cdot,\cdot)$ denotes binomial distribution. Let 
\[
\mu_{r, i}=\E_{\theta_N}\left[K_{M,\theta_N}^{r}(i) \mid \mathbf{W}_N, \mathbf{F}_{N,0}\right] ~~~ {\rm and}~~~  \mu_{r, s, i, j}=\E_{\theta_N}\left[K_{M,\theta_N}^{r}(i) K_{M,\theta_N}^{s}(j) \mid \mathbf{W}_N, \mathbf{F}_{N,0}\right],
\]
we know from the conditional distribution that
\begin{align*}
\mu_{2, i}=& N_{1}P_{N i}(1-P_{N i})+N_{1}^2P_{N i}^{2}.
\end{align*}
We first prove
\begin{align}\label{lem8_original}
&\frac{\ind(C_1N\leq N_0\leq C_2N)}{N_0} \sum_{i:W_i=0} \bar{m}(0,p(X_i;\theta_N)) \left(\left(\frac{K_{M,\theta_N}(i)}{M}\right)^{2}-\frac{\mu_{2,i}}{M^2}\right)=o_{\operatorname{P}_N}(1).&
\end{align}
Consider
\begin{align}
&\E_{\theta_N}\left[\left(\frac{\ind(C_1N\leq N_0\leq C_2N)}{N_{0}} \sum_{i:W_i=0} \bar{m}(0,p(X_i;\theta_N))\frac{K_{M,\theta_N}^{2}(i)-\mu_{2, i}}{M^2}\right)^{2} \mid \mathbf{W}_N, \mathbf{F}_{N,0}\right]&\notag\\
=&\frac{\ind(C_1N\leq N_0\leq C_2N)}{N_{0}^{2}} \sum_{i:W_i=0} \bar{m}(0,p(X_i;\theta_N))^2\frac{\mu_{4, i}-\mu_{2, i}^{2}}{M^4}&\notag\\
&+\frac{\ind(C_1N\leq N_0\leq C_2N)}{N_{0}^{2}} \sum_{(i,j):W_i=W_j=0,i\neq j}\bar{m}(0,p(X_i;\theta_N)) \bar{m}\left(0, p(X_j;\theta_N)\right)\frac{\mu_{2,2, i, j}-\mu_{2, i} \mu_{2, j}}{M^4}.\label{decompose}
\end{align}
For the first term on the right-hand side, since $\mathbf{W}_N$ only reveals information about $N_0$ and $W_{i}$ about $K_{M,\theta_N}(i)$, we know that
\begin{align*}
&\E_{\theta_N}\left[\frac{\ind(C_1N\leq N_0\leq C_2N)}{N_{0}^{2}} \sum_{i:W_i=0} \bar{m}^2\left(0, p(X_i;\theta_N)\right)\frac{\mu_{4, i}-\mu_{2, i}^{2}}{M^4}\mid \mathbf{W}_N\right]\\
\leq &\frac{\ind(C_1N\leq N_0\leq C_2N)C_{m}^2}{N_{0}}\E_{\theta_N}\left[\left(\frac{K_{M,\theta_N}(i)}{M}\right)^4\mid W_i=0, N_0\right],
\end{align*}
where $C_{m}^2$ is the upper bound of $|\bar{m}|$. From Lemma \ref{lem:7},  we know that the term is uniformly $o(1)$ in $N_0$. Accordingly, the expectation of the first term on the right-hand side of (\ref{decompose}) is $o(1)$.

For the second term, let $F_{(v)}$ be the $v$-th smallest value in $\mathbf{F}_{N,0}$. As long as $W_i=0$, let the catchment interval of observation $i$ be
\begin{align*}
    &A_M(i) := \left\{\begin{array}{cc}
			\left(\frac{F_{(v-M)}+F_{(v)}}{2}, \frac{F_{(v+M)}+F_{(v)}}{2}\right),\quad &M+1\leq v\leq N_0-M,\\
			\left(\frac{F_{(v-M)}+F_{(v)}}{2}, F_{(N_0)}\right),\quad &v> N_0-M,\\
                \left(F_{(1)}, \frac{F_{(v+M)}+F_{(v)}}{2}\right),\quad &v<M+1,
		\end{array}\right.&
\end{align*}
where $v$ is the rank of $p(X_i;\theta_N)$ among $\mathbf{F}_{N,0}$. Let  $I_{ij}$  be an indicator function that takes value 1 if the catchment intervals of observations  $i$  and  $j$ overlap, and value zero otherwise. According to Lemma S.10 in \cite{abadie2016matching}, only when $I_{ij}=1$, there may be $\left(\mu_{2,2, i, j}-\mu_{2, i} \mu_{2, j}\right)>0$, and the number of $(i,j)$ satisfying $i\neq j$ and $I_{ij}=1$ is smaller than  $2(M-1)N_{0}$. Thus,
\begin{align*}
&\frac{\ind(C_1N\leq N_0\leq C_2N)}{N_{0}^{2}} \sum_{(i,j):W_i=W_j=0,i\neq j} \bar{m}(0,p(X_i;\theta_N)) \bar{m}(0,p(X_j;\theta_N))\frac{\mu_{2,2, i, j}-\mu_{2, i} \mu_{2, j}}{M^4}\\
\leq&\frac{\ind(C_1N\leq N_0\leq C_2N)}{N_{0}^{2}} \sum_{(i,j):W_i=W_j=0,i\neq j}\bar{m}(0,p(X_i;\theta_N))\bar{m}(0,p(X_j;\theta_N))\frac{\mu_{4, i}+\mu_{4, j}}{2M^4} I_{ij}.
\end{align*}
We can go on to obtain that
\begin{align*}
&\E_{\theta_N}\left[\frac{\ind(C_1N\leq N_0\leq C_2N)}{N_{0}^{2}} \sum_{(i,j):W_i=W_j=0,i\neq j}\bar{m}(0,p(X_i;\theta_N))\bar{m}(0,p(X_j;\theta_N))\frac{\mu_{4, i}+\mu_{4, j}}{M^4} I_{ij}\mid \mathbf{W}_N\right]&\\
\leq&\frac{C_{m}^2\ind(C_1N\leq N_0\leq C_2N)}{N_{0}^{2}} \left(\E_{\theta_N}\left[\sum_{(i,j):W_i=W_j=0,i\neq j}\left(\frac{\mu_{4, i}+\mu_{4, j}}{M^4}\right)^2\mid \mathbf{W}_N\right] \E_{\theta_N}\left[\sum_{(i,j):W_i=W_j=0,i\neq j}I_{ij}\mid \mathbf{W}_N\right]\right)^{1/2}\\
\leq& \frac{\ind(C_1N\leq N_0\leq C_2N)C_{m}^2}{N_{0}^{2}} \left(4N_0^2\E_{\theta_N}\left[\frac{K_M(i)^8}{M^8}\mid W_i=0,N_0\right] 2(M-1)N_0\right)^{1/2}.
\end{align*}
Then, since 
\[
\frac{\ind(C_1N\leq N_0\leq C_2N)(M-1)}{N_0}\to0
\]
uniformly in $N_0$ and 
\[
\ind(C_1N\leq N_0\leq C_2N)\E_{\theta_N}\left[\frac{K_M(i)^8}{M^8}\mid W_i=0,N_0\right]
\]
is uniformly bounded in $N_0$, we obtain that the expectation of the second term on the right-hand side of (\ref{decompose}) is also $o(1)$. Since the left-hand side of (\ref{lem8_original}) converges to 0 in mean squares, (\ref{lem8_original}) holds. 

By Lemma \ref{lem:5} and Lemma \ref{lem:6}, we know that
\begin{align}
    &\frac{\ind(C_1N\leq N_0\leq C_2N)}{N_0} \sum_{i:W_i=0} \bar{m}(0,p(X_i;\theta_N)) \frac{N_1P_{Ni}}{M}=O_{\operatorname{P}_{\theta_N}}(1),
\end{align}
which, in combination with (\ref{P_NK}), yields that
\begin{align*}
    &\frac{\ind(C_1N\leq N_0\leq C_2N)}{N_0} \sum_{i:W_i=0} \bar{m}(0,p(X_i;\theta_N)) \frac{N_1P_{Ni}(1-P_{Ni})}{M^2}\overset{\operatorname{P}_N}{\to}0,
\end{align*}
since $M\to\infty$ and $\E_{\theta_N}\left[\frac{NP_{Ni}^2}{M^2}\mid W_i=0\right]\to0$ as $N\to\infty$. Therefore, (\ref{lem8_original}) is equivalent to what we want to prove.
\end{proof}
\begin{remark}
    Actually, this lemma still holds when we replace $\theta_N$ with any sequence $\tilde{\theta}_N\to\theta^*$. Furthermore, the convergences achieved by this lemma are uniform in a family of $\tilde{\theta}_N$ as long as $\tilde{\theta}_N\to\theta^*$ uniformly.
\end{remark}
{
\bibliographystyle{apalike}
\bibliography{AMS}
}

\end{document}